\documentclass[reqno]{amsart}
\usepackage[T1]{fontenc}
\usepackage[pdftex]{color,graphicx}
\usepackage{wrapfig} \usepackage[numbers,square]{natbib}
\usepackage{caption}
\title{Singularities of axisymmetric free surface flows with gravity}
\author[E. Varvaruca]{Eugen Varvaruca}
\address{Department of Mathematics and Statistics, University of Reading, Whiteknights, PO Box 220,
Reading RG6 6AX, U.K.}
\email{e.varvaruca@reading.ac.uk}
\author[G.S. Weiss]{Georg S. Weiss}
\address{Department of Mathematics, Heinrich Heine University, 40225 D\"usseldorf, Germany}
\email{weiss@math.uni-duesseldorf.de}

\date{}

\theoremstyle{plain}
\newtheorem{theorem}{Theorem}[section]

\newtheorem{lemma}[theorem]{Lemma}
\newtheorem{proposition}[theorem]{Proposition}

\theoremstyle{definition}
\newtheorem{definition}[theorem]{Definition}
\theoremstyle{definition}
\newtheorem{remark}[theorem]{Remark}
\numberwithin{equation}{section}

\def\R{{\bf R}}

\def\div{\textrm{\rm div }}\def\curl{\textrm{\rm curl }}

\def\dist{\textrm{\rm dist}}

\def\dh{\,d{\mathcal H}^1}
\def\dhone{\,d{\mathcal H}^1}

\newcommand{\Om}{\Omega}

\newcommand{\ipbrx}{\int_{\partial B_r^+(0)}}

\newcommand{\non}{\nonumber}
\newcommand{\be}{\begin{equation}}
\newcommand{\ee}{\end{equation}}
\newcommand{\we}{\frac{1}{x_1}}

\begin{document}

\maketitle

\begin{abstract}
We consider a steady axisymmetric solution of the Euler equations
for a fluid (incompressible and with zero vorticity) with a free surface,
acted on only by gravity.
We analyze stagnation points as well as points on the axis of
symmetry.
At points on the axis of
symmetry which are not stagnation points,
{\em constant velocity motion} is the only blow-up profile
consistent with the invariant scaling of the equation.
This suggests the presence of downward pointing cusps
at those points.

At {\em stagnation points} on the axis of
symmetry, the unique blow-up profile
consistent with the invariant scaling of the equation
is {\em Garabedian's pointed bubble solution} with water above
air. Thus at stagnation points on the axis of
symmetry with no water above the stagnation point,
the invariant scaling of the equation cannot be the right
scaling. A fine analysis of the blow-up velocity
yields that in the case that the surface is described
by an injective curve, the velocity scales almost like
$\sqrt{X^2+Y^2+Z^2}$ and is asymptotically given by the velocity field
$$V(\sqrt{X^2+Y^2},Z)=c (-\sqrt{X^2+Y^2},
2Z)$$
with a nonzero constant $c$.

The last result relies on a frequency formula
in combination with a concentration compactness result
for the axially symmetric Euler equations by J.-M. Delort.
While the concentration compactness result alone does
{\em not} lead to strong convergence in general, we prove
the convergence to be strong in our application.
\end{abstract}

\section{Introduction}
Consider the steady axisymmetric Euler equations for a fluid (incompressible and with
zero vorticity) with a free surface acted on only by gravity.
Using cylindrical coordinates and the Stokes stream function $\psi$
(see for example \cite[Exercise 4.18 (ii)]{fraenkel}),
we obtain the free boundary problem
\begin{align} \label{intro_sol}
&\div \left(\frac{1}{x_1} \nabla \psi(x_1,x_2)\right)=0 \textrm{ in the water phase } \{ \psi>0\}\\
&\frac{1}{x_1^2}{\vert \nabla \psi(x_1,x_2)\vert}^2
=  - x_2  \text{ on the free surface } \partial
\{\psi >0\};\non\end{align}
here the original velocity field
$$V(X,Y,Z) = \left(-{1\over {x_1}} \partial_2 \psi \cos \vartheta,
-{1\over {x_1}} \partial_2 \psi \sin \vartheta ,
{1\over {x_1}} \partial_1 \psi\right),$$
where $(X,Y,Z)=(x_1\cos\vartheta, x_1\sin\vartheta, x_2)$.

Observe that the positive sign of $\psi$ is chosen just for convenience and
that replacing $\psi$ by $-\psi$ our analysis covers the case of negative
$\psi$ as well.

Note also that the equations above describe
apart from a model, where the fluid is pumped in or sucked out
at a fixed boundary, also the case of a traveling wave traveling in the
direction of the axis of symmetry; here the equations describe the steady flow in the moving
frame, so that the
original velocity field is $$V(X,Y,Z, t) =\widetilde V(X,Y,Z-c_0t)+(0,0,c_0),$$ where $c_0$ is the speed of the traveling wave and
$$\widetilde V(X,Y,Z)=\left(-{1\over {x_1}}\partial_2 \psi \cos \vartheta,
-{1\over {x_1}} \partial_2 \psi\sin \vartheta,
{1\over {x_1}} \partial_1 \psi\right).$$
\cite{S-rev} and \cite{Sw1}, \cite{Sw2} are excellent reviews on two-dimensional water waves.

The free boundary problem (\ref{intro_sol}) has been studied in \cite{acf}
where regularity away from the degenerate sets $\{x_1=0\}$ (the axis of symmetry)
and $\{ x_2=0\}$ (containing all stagnation points)
has been shown for minimizers of a certain energy.

In the present paper we will focus on precisely those two sets and
analyze the profile of the velocity vector field close to points in those
sets.

Due to the degeneracy of the free boundary condition
$\vert \nabla \psi(x_1,x_2)\vert^2
=   {x_1}^2 x_2$ at points $x^0=(x_1^0,x_2^0)$ with $x_1^0 x_2^0=0$, we obtain {\em four}
invariant scalings
\begin{align*}
&{\psi(x^0+rx)\over r} \textrm{ in the case } x_1^0\ne 0 \textrm{ and } x_2^0\ne 0,\\
&{\psi(x^0+rx)\over {r^{3\over 2}}} \textrm{ in the case } x_1^0\ne 0 \textrm{ and } x_2^0=0,\\
&{\psi(x^0+rx)\over {r^{2}}} \textrm{ in the case } x_1^0=0 \textrm{ and } x_2^0\ne 0,\\
&{\psi(x^0+rx)\over {r^{5 \over 2}}} \textrm{ in the case } x_1^0=x_2^0=0.
\end{align*}
Note that the velocity (in the moving frame) would scale like $1,|x|^{{1\over 2}},1,|x|^{1\over 2}$
in the respective cases.

In a first main result we determine the profile
of the scaled solution as $r\to 0$ (Proposition \ref{2dim}):
In the case $x_1^0\ne 0$ and $x_2^0\ne 0$ the only asymptotics possible
is constant velocity flow parallel to the free surface.
In the case $x_1^0\ne 0$ and $x_2^0=0$ the only asymptotics possible
is the well-known Stokes corner flow (see \cite{toland}, \cite{plotnikov}, \cite{english}, \cite{VW}). Due to the perturbed equation the situation is actually not unlike
the two-dimensional problem in the presence of vorticity (see
\cite{VW2}, \cite{CS}, \cite{CS2}, \cite{CS3} for two-dimensional results in the
presence of vorticity).
In the case $x_1^0=0$ and $x_2^0\ne 0$ the only asymptotics possible
is constant velocity flow in the gravity direction. This suggests the {\em possibility
of air cusps} pointing in the gravity direction (Figure \ref{fig9}).

\begin{minipage}{\textwidth}
\begin{center}
\begin{picture}(0,0)%
\includegraphics{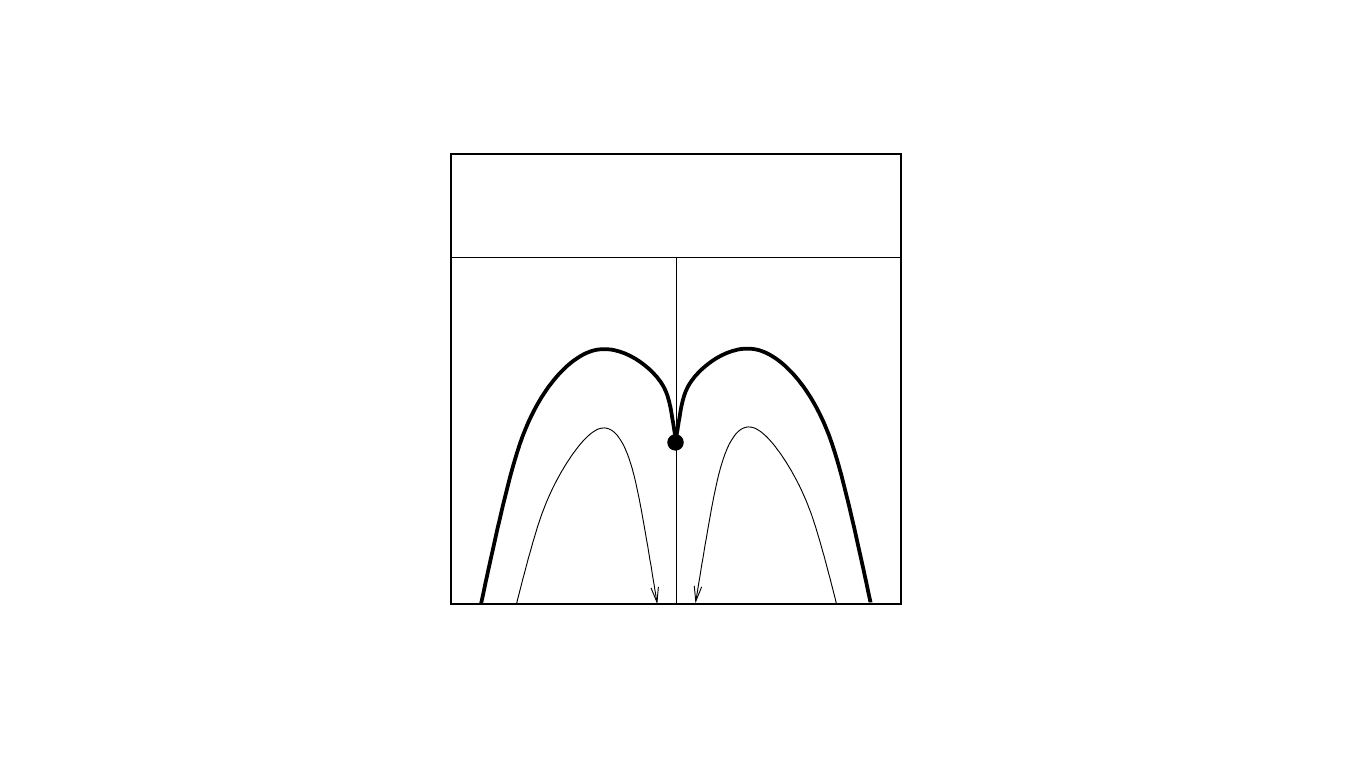}%
\end{picture}%
\setlength{\unitlength}{2368sp}%
\begingroup\makeatletter\ifx\SetFigFont\undefined%
\gdef\SetFigFont#1#2#3#4#5{%
  \reset@font\fontsize{#1}{#2pt}%
  \fontfamily{#3}\fontseries{#4}\fontshape{#5}%
  \selectfont}%
\fi\endgroup%
\begin{picture}(10834,6044)(-2411,-5183)
\put(2720,-1100){\makebox(0,0)[lb]{\smash{{\SetFigFont{7}{8.4}{\rmdefault}{\mddefault}{\updefault}{\color[rgb]{0,0,0}$(0,0)$}%
}}}}
\end{picture}%

\end{center}
\captionof{figure}{Dynamics suggested by our analysis}\label{fig9}
\end{minipage}

In the case $x_1^0=x_2^0=0$ the only asymptotics possible
is the {\em Garabedian pointed bubble solution} with water above
air (cf. \cite{garabedian}, Figure \ref{fig11}). This comes at first as a surprise
as it means that there is no nontrivial asymptotic profile at all
with air above water and with the invariant
scaling. However there remains at this stage the
possibility that the solution has a higher growth than that suggested
by the invariant scaling.

\begin{figure}
\begin{center}
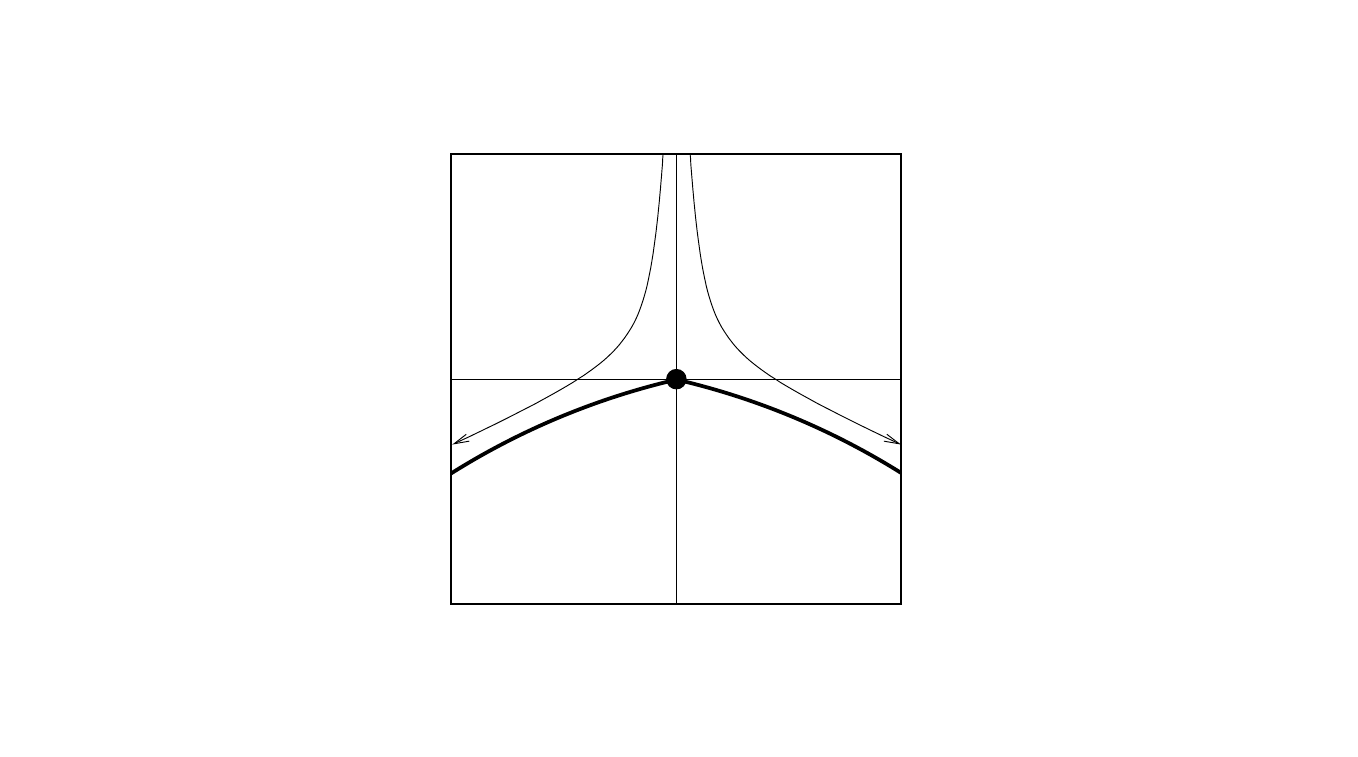
\end{center}
\captionof{figure}{Garabedian pointed bubble asymptotics}\label{fig11}
\end{figure}
In Theorem \ref{curve} we first analyze the possible shapes of
the surface close to stagnation points and close to points on the axis of symmetry.
Assuming that the surface is given by an injective curve
and assuming also a strict Bernstein inequality (corresponding to a Rayleigh-Taylor
condition) we obtain the following result:

In the case $x_1^0\ne 0$ and $x_2^0=0$ the only asymptotics possible
are the well-known Stokes corner (an angle of opening $120^\circ$ in the
direction of the axis of symmetry), and a horizontal point.

In the case $x_1^0=0$ and $x_2^0<0$ the only asymptotics possible
are cusps in the direction of the axis of symmetry.

In the case $x_1^0=x_2^0=0$ the only asymptotics possible
are the {\em Garabedian pointed bubble asymptotics} (an angle of opening
$\approx 114.799^\circ$ with water above air), and a {\em horizontal point}.

A fine analysis of the velocity profile in the last case ($x_1^0=x_2^0=0$
and a horizontal point) is no mean feat, and we confine ourselves to the
case of air above water. Here we prove (Theorem \ref{deg2d}) that
the velocity scales almost like
$\sqrt{X^2+Y^2+Z^2}$ and is asymptotically given by the velocity field
$$V(\sqrt{X^2+Y^2},Z)=c (-\sqrt{X^2+Y^2},
2Z),$$
where $c$ is a nonzero constant (Figure \ref{fig10}).

\begin{minipage}{\textwidth}
\begin{center}
\begin{picture}(0,0)%
\includegraphics{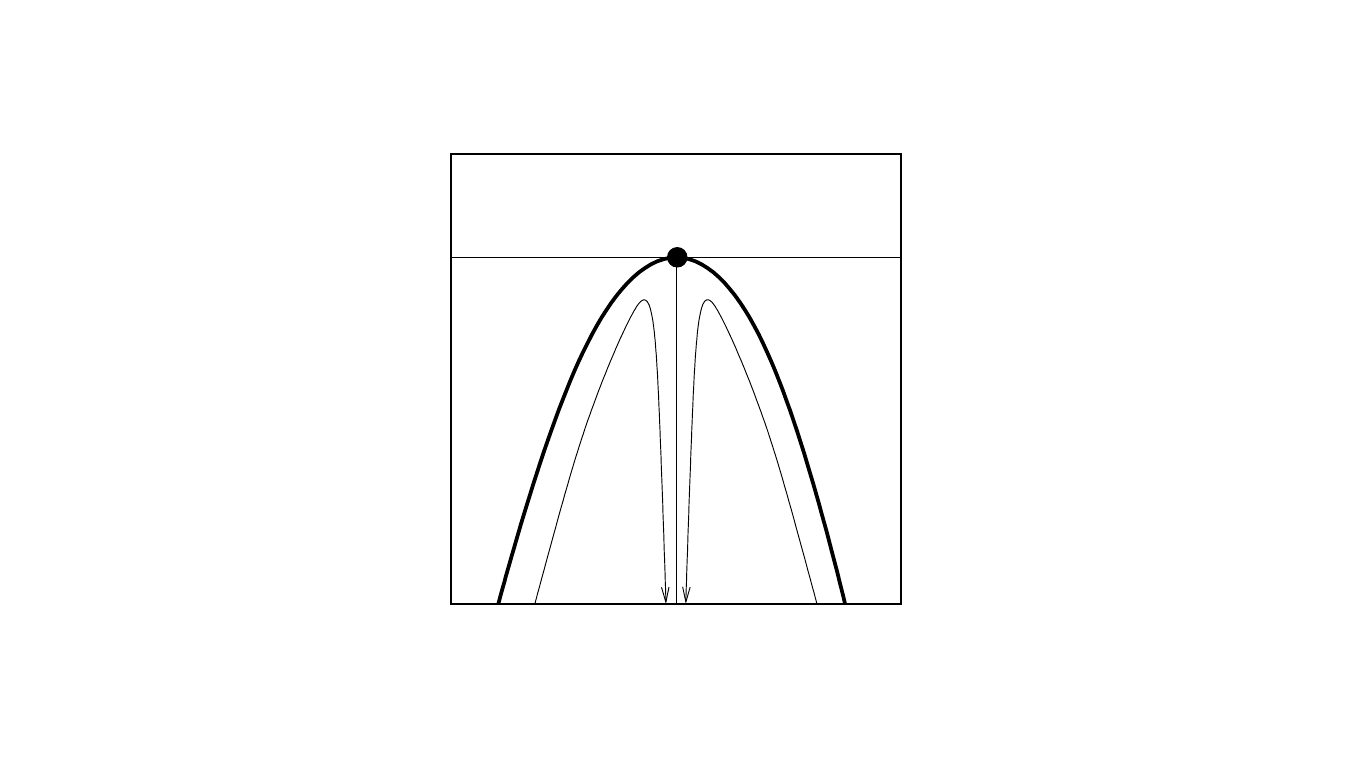}%
\end{picture}%
\setlength{\unitlength}{2368sp}%
\begingroup\makeatletter\ifx\SetFigFont\undefined%
\gdef\SetFigFont#1#2#3#4#5{%
  \reset@font\fontsize{#1}{#2pt}%
  \fontfamily{#3}\fontseries{#4}\fontshape{#5}%
  \selectfont}%
\fi\endgroup%
\begin{picture}(10834,6044)(-2411,-5183)
\put(2722,-1054){\makebox(0,0)[lb]{\smash{{\SetFigFont{7}{8.4}{\rmdefault}{\mddefault}{\updefault}{\color[rgb]{0,0,0}$(0,0)$}%
}}}}
\end{picture}%

\end{center}
\captionof{figure}{Dynamics suggested by our analysis}\label{fig10}
\end{minipage}

The proofs rely on a monotonicity formula as well as a {\em frequency formula}
for the axisymmetric problem;
as remarked in \cite{VW}, it is for certain semilinear problems possible
to derive {\em on the set of highest density} not a perturbation of Almgren's frequency formula
(see \cite{almgren}, \cite{lin2}, \cite{lin}, \cite{arshak}),
but a true nonlinear frequency formula. Here we extend the formula of \cite{VW} to the axisymmetric case.
In combination with a concentration compactness result
for the axially symmetric Euler equations by J.-M. Delort \cite{delort},
this leads to the already mentioned profile for the velocity vector field.
Note that
while the concentration compactness result alone does
{\em not} lead to strong convergence in general, we prove
the convergence to the limiting velocity vector field
to be strong in our application.
\section{Notation}
\label{notation}
We will use coordinates $(X,Y,Z)$ in the physical space
$\R^3$ together with partial derivatives $\partial_X,\partial_Y,\partial_Z$
as well as two-dimensional coordinates
$x=(x_1,x_2)$ together with partial derivatives $\partial_1,\partial_2$.
Sometimes we are going to to use cylindrical coordinates
$(X,Y,Z)=(x_1 \cos \vartheta, x_1 \sin \vartheta, x_2)$. We denote by $x\cdot y$ the Euclidean inner
product in $\R^n \times \R^n$, by $\vert x \vert$ the Euclidean norm
in $\R^n$, by $B_r(x^0):= \{x \in \R^n : \vert x-x^0 \vert < r\}$
the ball of center $x^0$ and radius $r$,
by $B^+_r(x^0):= \{x \in \R^n : x_1>0 \textrm{ and }\vert x-x^0 \vert < r\}$,
by $\partial B^+_r(x^0):= \{x \in \R^n : x_1>0 \textrm{ and }\vert x-x^0 \vert =r\}$
and $\R^n_+ := \{ (x_1,\dots,x_n) : x_1>0\}$ the positive parts.
\emph{Note that $\partial B^+_r(x^0)$ is not the topological boundary
of $B^+_r(x^0)$ and that $B^+_r(x^0)$ is not necessarily a half ball.}

We will use the notation
$B_r$ for $B_r(0)$ as well as $B^+_r$ for $B^+_r(0)$, and denote by $\omega_2$ the $2$-dimensional
volume of $B_1$.

We will use the weighted $L^p$ space
$$ L^p_w(\R^2_+):= \{ v \textrm{ measurable} :  \int_{\R^2_+} \we |v|^p\,dx <+\infty\}$$
with norm $\Vert f\Vert_{L^p_w(\R^2_+)} =  \left(\int_{\R^2_+} \we |v|^p\,dx\right)^{1\over p}$,
the weighted
Sobolev space
\begin{align*}W^{1,p}_w(\R^2_+):= &\{ v \in L^p_w(\R^2_+) : \textrm{ all weak partial derivatives of } v
\\ &\textrm{ are contained in } L^p_w(\R^2_+)\}\end{align*}
as well as the local spaces $$L^p_{w,loc}(\R^2_+)
:= \{ v \textrm{ measurable} :  \int_{B_R^+} \we |v|^p\,dx <+\infty
\textrm{ for each } R\in (0,+\infty)\}$$
and
\begin{align*}W^{1,p}_{w,loc}(\R^2_+)
:= & \{ v \textrm{ measurable} :  v \in L^p_{w,loc}(\R^2_+)
\textrm{ and all weak partial derivatives}\\ &\textrm{of } v
\textrm{ are contained in } L^p_w(B_R^+)
\textrm{ for each } R\in (0,+\infty)\}.\end{align*}

We denote by $\chi_A$ the characteristic function
of a set $A$. For any real number $a$, the notation $a^+$
stands for $\max(a, 0)$ and $a^-$
stands for $\min(a, 0)$.  Also, ${\mathcal L}^n$ shall denote the
$n$-dimensional Lebesgue measure and ${\mathcal H}^s$ the
$s$-dimensional Hausdorff measure. By $\nu$ we will always refer to
the outer normal on a given surface. We will use functions of
bounded variation $BV(U)$, i.e. functions $f\in L^1(U)$ for which
the distributional derivative is a vector-valued Radon measure. Here
$|\nabla f|$ denotes the total variation measure.
Note that for a smooth open set $E\subset \R^2$, $|\nabla \chi_E|$
coincides with the surface measure on $\partial E$.
We will also use the reduced boundary $\partial_{red} E$.

\section{Notion of solution and monotonicity formula}\label{notion}

Let
$\Omega$ be a bounded domain contained in $\{ (x_1,x_2): x_1\ge 0\}$, in which to consider
the combined problem for fluid and air. We study solutions $u$, in a
sense to be specified, of the problem
\begin{align} \label{strongp}
\div (\frac{1}{x_1} \nabla u)=
 \frac{\partial}{\partial x_1}\left(\we \frac{\partial u}{\partial x_1}\right)+\frac{\partial}{\partial x_2}\left(\we \frac{\partial u}{\partial x_2}\right) &=0 \quad\text{in } \Omega \cap \{u>0\},\\
\frac{1}{x_1^2}{\vert \nabla u\vert}^2
  &=   x_2  \quad\text{on } \Omega\cap\partial
\{u>0\}.\non\end{align} Note that, compared to the Introduction, we
have switched notation from $\psi$ to $u$, and we have ``reflected''
the problem at the hyperplane $\{x_2=0\}$. Since our results are
completely local, we do not specify boundary conditions on
$\partial\Omega$.

We begin by introducing our notion of a {\em variational solution}
of problem (\ref{strongp}).
\begin{definition}[Variational Solution]
\label{vardef} We define $u\in W^{1,2}_{w,\textnormal{loc}}(\Omega)$
to be a {\em variational solution} of (\ref{strongp}) if $u\in
C^0(\Omega)\cap C^2(\Omega \cap \{ u>0\})$, $\nabla u/x_1\in
C^1(\Omega \cap \{ u>0\})$, $u=0$ on $\{ x_1=0\}$ (motivated by
the fact that the velocity on the axis orthogonal to the axis direction should be zero),
$u\geq 0$ in $\Omega$,
 and the first variation with respect to
domain variations of the functional
\[ J(v) :=  \int_\Omega \left(\we {\vert \nabla v \vert}^2 +
x_1 x_2  \chi_{\{ v>0\}} \right)\,dx\] vanishes at $v=u  ,$ i.e.
\begin{align} 0 &= -\frac{d}{d\epsilon} J(u(x+\epsilon
\phi(x)))|_{\epsilon=0}  \non\\&= \int_\Omega \Bigg[ \left(\we{\vert \nabla u
\vert}^2 +x_1x_2 \chi_{\{ u>0\}}\right)\div\phi-2\we \nabla u D\phi\nabla u \non\\&\qquad+ \left( -\frac{1}{x_1^2}\vert\nabla u\vert^2+ x_2\chi_{\{
u>0\}}\right)\phi_1 +x_1\chi_{\{ u>0\}} \phi_2\Bigg]\,dx\non\end{align} for any
$\phi=(\phi_1,\phi_2) \in C^1_0(\Omega;\R^2) \textrm{ such that } \phi_1=0 \textrm{ on } \{ x_1=0\}.$
\end{definition}
A proof of the just mentioned first variation formula
can be found in \cite[Section 3.2]{GH}. An integration by parts shows that
$u$ satisfies on smooth parts of free boundary $\partial\{u>0\}$ in $\{ x_1x_2\ne 0\}$
the free boundary condition
$$\frac{1}{x_1}{\vert \nabla u\vert}^2
  =   x_1x_2.$$

\begin{theorem}[Monotonicity Formula]\label{elmon2}
Let $u$ be a variational solution of (\ref{strongp}), let $x^0\in\Om$ and let
$\delta:=\textnormal{dist}(x^0,\partial\Om)/2$.
Let, for any $r\in (0,\delta)$,
\begin{align} &I(r)=\int_{B^+_r(x^0)}\left(\we|\nabla u|^2+x_1x_2\chi_{\{u>0\}}\right)\,dx,\label{I}\\
 &J(r)=\int_{\partial B^+_r(x^0)}\we u^2\dh,\label{J}\\
&M^{int}(r)=r^{-2}I(r)- r^{-3}J(r),\label{Mint}\\
&M^{x_2}(r)=r^{-3}I(r)- \frac{3}{2} r^{-4}J(r),\label{Mx2}\\
&M^{x_1}(r)=r^{-3}I(r)- 2 r^{-4}J(r),\label{Mx1}\\
&M^{x_1x_2}(r)=r^{-4}I(r)- \frac{5}{2} r^{-5}J(r).\label{Mx1x2}
\end{align}
Then, for a.e. $r\in (0,\delta)$,
\begin{align}
\label{Mprint}&(M^{int}(r))'
= 2r^{-2}\int_{\partial B^+_r(x^0)}\we\left(\nabla u\cdot \nu-\frac{u}{r}\right)^2\dh\\ &\quad\quad
+ r^{-3}\int_{B^+_r(x^0)}\left( -\frac{x_1-x^0_1}{x_1^2}|\nabla u|^2 + \left[
(x_1-x^0_1)x_2 + (x_2-x^0_2)x_1\right]\chi_{\{
u>0\}}\right) \,dx\non \\ &\quad\quad
+ r^{-4} \int_{\partial B^+_r(x^0)} \frac{x_1-x^0_1}{(x_1)^2} u^2    \dh.\non
\intertext{In the case $x^0_2=0$,}
\label{Mprx2}&(M^{x_2}(r))'
= 2r^{-3}\int_{\partial B^+_r(x^0)}\we\left(\nabla u\cdot \nu-\frac{3}{2}\frac{u}{r}\right)^2\dh\\ &\quad\quad\non
+ r^{-4} \int_{B^+_r(x^0)}\left( -\frac{x_1-x^0_1}{x_1^2}|\nabla u|^2 +
(x_1-x^0_1)x_2 \chi_{\{
u>0\}}\right) \,dx\non\\ &
\quad\quad + \frac{3}{2}r^{-5} \int_{\partial B^+_r(x^0)} \frac{x_1-x^0_1}{(x_1)^2} u^2    \dh\non.
\intertext{In the case $x^0_1=0$,}
\label{Mprx1}&(M^{x_1}(r))'
= 2r^{-3}\int_{\partial B^+_r(x^0)}\we\left(\nabla u\cdot \nu-2\frac{u}{r}\right)^2\dh\\ &\quad\quad\non
+ r^{-4}\int_{B^+_r(x^0)}
(x_2-x^0_2)x_1 \chi_{\{
u>0\}} \,dx.
\intertext{Last, in the case $x^0_1=x^0_2=0$,}
\label{Mprx1x2}&(M^{x_1x_2}(r))'
=2r^{-4}\ipbrx\we\left(\nabla u\cdot \nu-\frac{5}{2}\frac{u}{r}\right)^2\dh.
\end{align}
\end{theorem}
\begin{remark}
\label{elrem}
(i) The integrand in the first
integral on the
right-hand side of (\ref{Mprint}) is a scalar multiple of
$ (\nabla u(x) \cdot (x-x^0) - u(x))^2$, and therefore vanishes
if and only if $u$ is a homogeneous function of degree $1$ with respect to $x^0$.
\\
(ii)
The integrand in the first
integral on the
right-hand side of (\ref{Mprx2}) is a scalar multiple of
$ (\nabla u(x) \cdot (x-x^0) - \frac{3}{2}u(x))^2$, and therefore vanishes
if and only if $u$ is a homogeneous function of degree $3/2$ with respect to $x^0$.
\\
(iii)
The integrand in the first
integral on the
right-hand side of (\ref{Mprx2}) is a scalar multiple of
$ (\nabla u(x) \cdot (x-x^0) - 2u(x))^2$, and therefore vanishes
if and only if $u$ is a homogeneous function of degree $2$ with respect to $x^0$.
\\
(iv)
The integrand in the first
integral on the
right-hand side of (\ref{Mprx2}) is a scalar multiple of
$ (\nabla u(x) \cdot x - \frac{5}{2}u(x))^2$, and therefore vanishes
if and only if $u$ is a homogeneous function of degree $5/2$.
\end{remark}
\begin{proof}
First,
for each $u\in W^{1,2}_{w,loc}(\R^2)$, each $\alpha \in\R$ and
a.e. $r\in(0,\delta)$ we obtain, setting $w_r(x):= u(x^0+rx)$,
\begin{align}
& \frac{d}{dr} \left(r^\alpha \int_{\partial B^+_r(x^0)} \we u^2
\dh\right)
 = \frac{d}{{dr}} \left(r^{\alpha+n-1}
\int_{\partial B^+_1} \frac{1}{x^0_1+rx_1}w_r^2    \dh\right)\label{easy}\\ \non &=
(\alpha+n-1) r^{\alpha-1} \int_{\partial B^+_r(x^0)} \we u^2    \dh
- r^{\alpha+n-1}
\int_{\partial B^+_1} \frac{x_1}{(x^0_1+rx_1)^2}w_r^2    \dh\\ \non &\quad
+ r^{\alpha+n-1} \int_{\partial B^+_1} \frac{2}{x^0_1+rx_1}w_r\nabla u(x^0+rx)\cdot x   \dh\\ \non
&= (\alpha+n-1) r^{\alpha-1} \int_{\partial B^+_r(x^0)} \we u^2    \dh -
r^{\alpha-1} \int_{\partial B^+_r(x^0)} \frac{x_1-x^0_1}{(x_1)^2} u^2    \dh\\ &\quad
+r^{\alpha} \int_{\partial B^+_r(x^0)} \frac{2}{x_1}u\nabla u\cdot \nu    \dh.\non
\end{align}

Suppose now that $u$ is a variational solution of (\ref{strongp}).
For small positive $\kappa$ and $\eta_\kappa(t) :=
\max(0,\min(1,\frac{r- t}{ \kappa}))$, we take after approximation
$\phi_\kappa(x) := \eta_\kappa(|x-x^0|)(x-x^0) $ as a test function
in the definition of a variational solution, obtaining
\begin{align}\non 0 =& \int_{\Om}
\left(\we |\nabla u |^2+x_1x_2\chi_{\{ u>0\}}\right)\left(2\eta_\kappa(|x-x^0|)+\eta_\kappa'(|x-x^0|)|x-x^0|\right)\,dx\\ \non
&-\int_{\Om}\frac{2}{x_1}\bigg[(\partial_1 u)^2\left(\eta_\kappa(|x-x^0|)+\eta_\kappa'(|x-x^0|)\frac{(x_1-x^0_1)^2}{|x-x^0|}\right)\\ \non
&+(\partial_2 u)^2\left(\eta_\kappa(|x-x^0|)+\eta_\kappa'(|x-x^0|)\frac{(x_2-x^0_2)^2}{|x-x^0|}\right)\non\\&\qquad\qquad\qquad +(\partial_1 u)(\partial_2 u)2\eta'(|x-x^0|)\frac{(x_1-x^0_1)(x_2-x^0_2)}{|x-x^0|}\bigg]\,dx\non\\& \non
+\int_{\Om}\bigg(-\frac{x_1-x^0_1}{x_1^2}|\nabla u|^2\eta_k(|x-x^0|)+\big[
(x_1-x^0_1)x_2 \\ &\qquad\qquad\qquad + (x_2-x^0_2)x_1\big]\chi_{\{
u>0\}}\eta_\kappa(|x-x^0|)\bigg)\,dx.\non\end{align}
Passing to the limit as $\kappa\to 0$, we
obtain for a.e. $r \in (0,\delta)$,
\begin{align}\label{mff2} 0=& 2\int_{B^+_r(x^0)} x_1x_2  \chi_{\{ u>0\}}
\,dx - r \int_{\partial B^+_r(x^0)}
 \left(\we {\vert \nabla u \vert}^2 + x_1x_2  \chi_{\{ u>0\}} \right)\dh\\ &\quad
+ 2r\int_{\partial B^+_r(x^0)}
\we (\nabla u \cdot \nu)^2\,\dh\non
\\ &
+ \int_{B^+_r(x^0)}\bigg( -\frac{x_1-x^0_1}{x_1^2}|\nabla u|^2 + \big[
(x_1-x^0_1)x_2 + (x_2-x^0_2)x_1\big]\chi_{\{
u>0\}}\bigg) \,dx.\non
\end{align}

Observe that letting $\epsilon\to 0$ in \[\int_{B^+_r(x^0)} \we \nabla u
\cdot \nabla \max(u-\epsilon,0)^{1+\epsilon}\,dx =\int_{\partial
B^+_r(x^0)} \we \max(u-\epsilon,0)^{1+\epsilon} \nabla u \cdot \nu\, \dh\]
for a.e. $r \in (0,\delta)$, we obtain the integration by parts
formula
\begin{equation}\label{part2}
 \int_{B^+_r(x^0)} \we {\vert \nabla u \vert}^2\,dx =\int_{\partial B^+_r(x^0)}
\we u \nabla u \cdot \nu\,  \dh \end{equation} for a.e. $r \in
(0,\delta).$

Note that
\begin{align*}
(r^{-2}I(r))'=& -2 r^{-3} \int_{B^+_r(x^0)} \left(\we|\nabla u|^2
+ x_1 x_2 \chi_{\{ u>0\}}\right) \,dx\non\\ &\qquad+r^{-2}  \int_{\partial B^+_r(x^0)}\left(\we|\nabla u|^2 + x_1x_2 \chi_{\{ u>0\}}\right)\dh,
\end{align*}
so that by (\ref{mff2}) and (\ref{part2}),
\begin{align}(r^{-2}I(r))'&
= r^{-3}\Bigg(2r\int_{\partial B^+_r(x^0)}\we(\nabla u\cdot \nu)^2\dh- 2\int_{\partial B^+_r(x^0)}\we u \nabla u\cdot \nu\dh\label{Iprint}\\ &
+ \int_{B^+_r(x^0)}\left( -\frac{x_1-x^0_1}{x_1^2}|\nabla u|^2 + \left[
(x_1-x^0_1)x_2 + (x_2-x^0_2)x_1\right]\chi_{\{
u>0\}}\right) \,dx
\Bigg),\non
\end{align}
Combining (\ref{Iprint}) and (\ref{easy}) with $\alpha=-3$ yields (\ref{Mprint}).

Moreover,
\begin{align}
(r^{-3}I(r))'& =-3 r^{-4} \int_{B^+_r(x^0)} \left(\we|\nabla u|^2
+ x_1 x_2 \chi_{\{ u>0\}}\right) \,dx\label{oldidx1}\\ & \qquad+r^{-3}  \int_{\partial B^+_r(x^0)}\left(\we|\nabla u|^2 + x_1x_2 \chi_{\{ u>0\}}\right)\dh.\non
\end{align}
In the case $x^0_2=0$ we obtain from (\ref{oldidx1}), using (\ref{mff2}) and (\ref{part2}),
that
\begin{align}
(r^{-3}I(r))'&=r^{-4}\Bigg(2r\int_{\partial B^+_r(x^0)}\we(\nabla u\cdot \nu)^2\dh- 3\int_{\partial B^+_r(x^0)}\we u \nabla u\cdot \nu\dh\label{Iprx2}\\ &
+ \int_{B^+_r(x^0)}\left( -\frac{x_1-x^0_1}{x_1^2}|\nabla u|^2 +
(x_1-x^0_1)x_2 \chi_{\{
u>0\}}\right) \,dx
\Bigg),\non
\end{align}
Combining (\ref{Iprx2}) and (\ref{easy}) with $\alpha=-4$ yields (\ref{Mprx2}).
On the other hand, in the case $x^0_1=0$ we obtain from (\ref{oldidx1}), using (\ref{mff2}) and (\ref{part2}),
that
\begin{align}
(r^{-3}I(r))'&=r^{-4}\Bigg(2r\int_{\partial B^+_r(x^0)}\we(\nabla u\cdot \nu)^2\dh- 4\int_{\partial B^+_r(x^0)}\we u \nabla u\cdot \nu\dh\label{Iprx1}\\
& + \int_{B^+_r(x^0)}
(x_2-x^0_2)x_1 \chi_{\{
u>0\}} \,dx
\Bigg),\non
\end{align}
Combining (\ref{Iprx1}) and (\ref{easy}) with $\alpha=-4$ yields (\ref{Mprx1}).

Last,
in the case $x^0_1=x^0_2=0$, since
\begin{align*}
(r^{-4}I(r))' &=-4 r^{-5} \int_{B^+_r(0)} \left(\we|\nabla u|^2 + x_1 x_2 \chi_{\{ u>0\}}\right) \,dx\non\\&\qquad+r^{-4}  \int_{\partial B^+_r(0)}\left(\we|\nabla u|^2 + x_1x_2 \chi_{\{ u>0\}}\right)\dh,\end{align*}
we obtain from (\ref{mff2}) and (\ref{part2}) that
\be(r^{-4}I(r))'= r^{-5}\left(2r\ipbrx\we(\nabla u\cdot \nu)^2\dh- 5\ipbrx\we u (\nabla u\cdot \nu)\dh \right).\label{Iprx1x2}\ee
Combining (\ref{Iprx1x2}) and (\ref{easy}) with $\alpha=-5$ yields (\ref{Mprx1x2}).
\end{proof}

\begin{lemma}[Bernstein estimate]
In $\{ u>0\}$, the solution satisfies
$$\Delta\left(\frac{|\nabla u|^2}{x_1}- x_1 x_2\right)=
2\sum_{i,j=1}^2 \frac{(\partial_{ij} u)^2}{x_1}.$$
\end{lemma}
\begin{proof}
Direct calculation.
\end{proof}
\begin{remark}
Constructing barrier solutions it is therefore possible to
verify $\frac{|\nabla u|^2}{x_1}- x_1 x_2\le 0$ for certain domains
$\subset \{ x_2>0\}$,
certain Dirichlet boundary data and the {\em minimal solution} $u$
(cf. \cite{ejde}).
\end{remark}
\begin{definition}[Weak Solution]\label{weak}
We define $u\in W^{1,2}_{w,\textnormal{loc}}(\Omega)$ to be a {\em
weak solution} of (\ref{strongp}) if the following are satisfied:
$u$ is a {\em variational solution} of (\ref{strongp})
and
the topological free boundary $\partial \{ u>0\}
\cap \Omega^\circ  \cap \{ x_2 \ne 0\}$ is locally
a
$C^{2,\alpha}$-surface.
\end{definition}
\begin{remark}
(i) It follows that in $\Omega^\circ \cap \{ x_2\ne 0\}$ the solution is a classical solution
of (\ref{strongp}). It follows also that $\partial\{ u>0\}\subset \{ x_2\ge 0\}$.

(ii) For any weak
solution $u$ of {\rm (\ref{strongp})} such that  \[\frac{|\nabla u|^2}{x_1}\le Cx_1 |x_2|\quad\text{locally in }\Omega,\]
$u$ is a variational
solution of {\rm (\ref{strongp})},
$ \chi_{\{ u>0\}}$ is
locally in $\Omega^\circ\cap \{ x_2> 0\}$ a function of bounded variation, and the
total variation measure $|\nabla \chi_{\{ u>0\}}|$ satisfies
\[\int_{B^+_r(x^0)} \sqrt{x^+_2}\,d|\nabla \chi_{\{
u>0\}}|\le C_1 \left\{\begin{array}{l} r^{3\over 2}, x^0_2=0\\
r \sqrt{x^0_2}, x^0_2> 0\end{array}\right.\]
for all $B^+_r(x^0)\subset\subset \Omega$.
The reason is that, integrating by parts,
\begin{align*}0 & = \int_{B^+_r(x^0)\cap \{ u>0\}} \div (\we \nabla u)
\le  \int_{\partial B^+_r(x^0)\cap \{ u>0\}}\frac{|\nabla u|}{x_1}\dh\\&\quad 
+ \int_{B_r(x^0)\cap \{ x_1=0\}} \frac{|\nabla u|}{x_1}\dh
-\int_{B^+_r(x^0)\cap \partial_{\rm red} \{ u>0\}} \frac{|\nabla u|}{x_1}\dh \\&
\le C_1 \left(\int_{\partial B^+_r(x^0)}\sqrt{x^+_2}\dh + \int_{B_r(x^0)\cap \{ x_1=0\}}\sqrt{x^+_2}\dh\right)
\\&\quad
-\int_{B^+_r(x^0)\cap \partial_{\rm red} \{ u>0\}}\sqrt{x^+_2}\dh.
\end{align*}
\end{remark}
\begin{lemma}\label{density_1}
Let $u$ be a variational solution of {\rm (\ref{strongp})} and
suppose that
\[\frac{|\nabla u|^2}{x_1}\le C  x_1 |x_2|\quad\text{locally in }\Omega.\]
Then:

(i) The limit $M^{int}(0+)=\lim_{r\to 0+} M^{int}(r)$ exists and is finite.
If $x^0_2=0$, then the limit $M^{x_2}(0+)=\lim_{r\to 0+} M^{x_2}(r)$ exists and is finite.
If $x^0_1=0$, then the limit $M^{x_1}(0+)=\lim_{r\to 0+} M^{x_1}(r)$ exists and is finite.
If $x^0_1=x^0_2=0$, then the limit $M^{x_1x_2}(0+)=\lim_{r\to 0+} M^{x_1x_2}(r)$ exists and is finite.

(iii)  Let $x^0_1>0$, $x^0_2>0$ and $0<r_m\to 0+$ as $m\to
\infty$ be a sequence such that the {\em blow-up} sequence
\begin{equation}\label{blo1} u_m(x) := {u(x^0+{r_m}x)/ {r_m}}\end{equation} converges weakly in
 $W^{1,2}_{\textnormal{loc}}(\R^2)$ to
a blow-up limit $u_0$. Then $u_0$ is a homogeneous function of
degree $1$, i.e. $u_0(\lambda x)  = \lambda u_0(x)$.

Let $x^0_2=0$ and let $0<r_m\to 0+$ as $m\to
\infty$ be a sequence such that the {\em blow-up} sequence
\begin{equation}\label{blo2}u_m(x) := {u(x^0+{r_m}x)/ {r_m}^{3\over 2}}\end{equation} converges weakly in
 $W^{1,2}_{\textnormal{loc}}(\R^2)$ to
a blow-up limit $u_0$. Then $u_0$ is a homogeneous function of
degree $3/2$.

Let $x^0_1=0$ and let $0<r_m\to 0+$ as $m\to
\infty$ be a sequence such that the {\em blow-up} sequence
\begin{equation}\label{blo3}u_m(x) := {u(x^0+{r_m}x)/ {r_m}^2}\end{equation} converges weakly in
 $W^{1,2}_{w,\textnormal{loc}}(\R^2_+)$ to
a blow-up limit $u_0$. Then $u_0$ is a homogeneous function of
degree $2$.

Let $x^0_1=x^0_2=0$ and let $0<r_m\to 0+$ as $m\to
\infty$ be a sequence such that the {\em blow-up} sequence
\begin{equation}\label{blo4}u_m(x) := {u(x^0+{r_m}x)/ {r_m}^{5\over 2}}\end{equation} converges weakly in
 $W^{1,2}_{w,\textnormal{loc}}(\R^2_+)$ to
a blow-up limit $u_0$. Then $u_0$ is a homogeneous function of
degree $5/2$.

(iii) Let $u_m$ be one of the converging sequences in (ii). Then $u_m$
converges {\em strongly} in $W^{1,2}_{w,\textnormal{loc}}(\R^2_+)$
({\em strongly} in $W^{1,2}_{\textnormal{loc}}(\R^2)$ in the cases where $x^0_1>0$).

(iv) If $x^0_1>0$ and $x^0_2\ne 0$, then
$$M^{int}(0+)=x^0_1 x^0_2 \lim_{r\to 0+} r^{-2}
\int_{B^+_r(x^0)}\chi_{\{ u>0\}}\,dx.$$
Moreover, $M^{int}(0+)=0$ implies that $u_0=0$ in $\R^2$ for
each blow-up limit $u_0$ of $u_m(x) = u(x^0+{r_m}x)/r_m$.

If $x^0_1>0$ and $x^0_2=0$, then
$$M^{x_2}(0+)=x^0_1 \lim_{r\to 0+} r^{-3}
\int_{B^+_r(x^0)}x_2 \chi_{\{ u>0\}}\,dx.$$

If $x^0_1=0$ and $x^0_2\ne 0$, then
$$M^{x_1}(0+)=x^0_2 \lim_{r\to 0+} r^{-3}
\int_{B^+_r(x^0)}x_1 \chi_{\{ u>0\}}\,dx.$$
Moreover, $M^{x_1}(0+)=0$ implies that $u_0=0$ in $\R^2_+$ for
each blow-up limit $u_0$ of $u_m(x) = u(x^0+{r_m}x)/{r_m}^2$.

If $x^0_1=x^0_2=0$, then
$$M^{x_1x_2}(0+)=\lim_{r\to 0+} r^{-4}
\int_{B^+_r(x^0)}x_1 x_2 \chi_{\{ u>0\}}\,dx.$$
\end{lemma}
\begin{proof}
(i) follows from the assumption
\[\frac{|\nabla u|^2}{x_1}\le C  x_1 |x_2|\quad\text{locally in }\Omega\]
together with Theorem \ref{elmon2}.

(ii): For each $0<\sigma<\infty$ the sequence $u_m$ is in each case by
assumption bounded in $C^{0,1}(B^+_\sigma)$ (bounded in $C^{0,1}(B_\sigma)$
in the case that $x^0_1>0$).
For any $0<\tau<\sigma<\infty$, we write the identities (\ref{Mprint}), (\ref{Mprx2}), (\ref{Mprint}), (\ref{Mprx1x2}) in integral form as
\begin{align}
&2\int_\tau^\sigma r^{-2}
\int_{\partial B^+_r(x^0)}{1\over {x_1}}
\left(\nabla u\cdot \nu -\frac{u}{r}\right)^2\dh dr\non \\
&=M^{int}(\sigma)-M^{int}(\tau)-\int_\tau^\sigma K^{int}(r)\,dr\label{jhmint}
\textrm{ in the case } x^0_1>0 \textrm{ and } x^0_2> 0,\\
&2\int_\tau^\sigma r^{-3}
\int_{\partial B^+_r(x^0)}{1\over {x_1}}
\left(\nabla u\cdot \nu -\frac{3}{2}\frac{u}{r}\right)^2\dh dr\non \\
&=M^{x_2}(\sigma)-M^{x_2}(\tau)-\int_\tau^\sigma K^{x_2}(r)\,dr\label{jhmx2}
\textrm{ in the case } x^0_1>0 \textrm{ and } x^0_2=0,
\end{align}
\begin{align}
&2\int_\tau^\sigma r^{-3}
\int_{\partial B^+_r(x^0)}{1\over {x_1}}
\left(\nabla u\cdot \nu -2\frac{u}{r}\right)^2\dh dr\non \\
&=M^{x_1}(\sigma)-M^{x_1}(\tau)-\int_\tau^\sigma K^{x_1}(r)\,dr\label{jhmx1}
\textrm{ in the case } x^0_1=0 \textrm{ and } x^0_2> 0,\\
&2\int_\tau^\sigma r^{-4}
\int_{\partial B^+_r(x^0)}{1\over {x_1}}
\left(\nabla u\cdot \nu -\frac{5}{2}\frac{u}{r}\right)^2\dh dr\non \\
&=M^{x_1x_2}(\sigma)-M^{x_1x_2}(\tau)\,dr\label{jhmx1x2}
\textrm{ in the case } x^0_1=x^0_2=0;\end{align}
here $K^{int}, K^{x_2}$ and $K^{x_1}$ are defined by (\ref{jhmint}), (\ref{jhmx2}) and (\ref{jhmx1}),
and they are all integrable.

It follows by rescaling in (\ref{jhmint})-(\ref{jhmx1x2}) that
\begin{align*}
2\int_{B_\sigma(0)\setminus B_\tau(0)}&|x|^{-3}{1\over {x_1}}\left(\nabla u_m(x) \cdot x -
u_m(x)\right)^2\,dx\\&\leq M^{int}(r_m\sigma)-M^{int}(r_m\tau)+\int_{r_m\tau}^{r_m\sigma} |K^{int}(r)|\,dr \to 0\quad\text{as }m\to\infty,\\
&\textrm{ in the case } x^0_1>0 \textrm{ and } x^0_2> 0,\\
2\int_{B_\sigma(0)\setminus B_\tau(0)}&|x|^{-5}{1\over {x_1}}\left(\nabla u_m(x) \cdot x - \frac{3}{2}
u_m(x)\right)^2\,dx\\&\leq M^{x_2}(r_m\sigma)-M^{x_2}(r_m\tau)+\int_{r_m\tau}^{r_m\sigma} |K^{x_2}(r)|\,dr \to 0\quad\text{as }m\to\infty,\\
&\textrm{ in the case } x^0_1>0 \textrm{ and } x^0_2=0,\\
2\int_{B^+_\sigma(0)\setminus B^+_\tau(0)}&|x|^{-5}{1\over {x_1}}\left(\nabla u_m(x) \cdot x - 2
u_m(x)\right)^2\,dx\\&\leq M^{x_1}(r_m\sigma)-M^{x_1}(r_m\tau)+\int_{r_m\tau}^{r_m\sigma} |K^{x_1}(r)|\,dr \to 0\quad\text{as }m\to\infty,\\
&\textrm{ in the case } x^0_1=0 \textrm{ and } x^0_2> 0,\\
2\int_{B^+_\sigma(0)\setminus B^+_\tau(0)}&|x|^{-6}{1\over {x_1}}\left(\nabla u_m(x) \cdot x - \frac{5}{2}
u_m(x)\right)^2\,dx\\&\leq M^{x_1x_2}(r_m\sigma)-M^{x_1x_2}(r_m\tau)\to 0\quad\text{as }m\to\infty\\
&\textrm{ in the case } x^0_1=x^0_2=0,
\end{align*} which yields the desired homogeneity of $u_0$.

(iii): In order to show
strong convergence of $\nabla u_m$,
it is in view of the weak $L^2_w$-convergence of $\nabla
u_m$ sufficient to prove convergence of the $L^2_w$-norm.

Let $\delta:=\text{dist}(x^0,\partial\Om)/2$. Then, for each $m$, $u_m$ is a variational solution of
\begin{align} \label{strongpm}
 \div\left( \frac{\nabla u_m(x)}{(x^0+r_mx)_1}\right) &=0\quad\left\{
\begin{array}{l}\textrm{in } B_{\delta/r_m} \cap \{u_m>0\} \textrm{ in the case } x^0_1>0,\\
\textrm{in } B^+_{\delta/r_m} \cap \{u_m>0\} \textrm{ in the case } x^0_1=0.
\end{array}\right.
\end{align}
Since $u_m$ converges to $u_0$ locally uniformly, it follows from (\ref{strongpm}) that $u_0$ is harmonic in $\{ u_0>0\}$ in the case $x^0_1>0$ and a solution of the equation
$$\div\left({1\over {x_1}}\nabla u_0\right)=0$$
in the case $x^0_1=0$.
Also, using the uniform convergence, the continuity of
$u_0$ and its solution property in $\{ u_0>0\}$
we obtain as in the proof of (\ref{part2}) that
\begin{align*}&o(1)+\int_{\R^2}
 \frac{1}{x^0_1}|\nabla u_m|^2\eta\,dx
\\&=\int_{\R^2}
 \frac{1}{(x^0+r_mx)_1}|\nabla u_m|^2\eta\,dx =
-\int_{\R^2} u_m \frac{1}{(x^0+r_mx)_1} \nabla u_m\cdot \nabla \eta \,dx
\\&\to
-\int_{\R^2} u_0 \frac{1}{x^0_1}\nabla u_0\cdot \nabla \eta\,dx =\frac{1}{x^0_1}
\int_{\R^2} |\nabla
u_0|^2\eta\,dx\textrm{ in the case } x^0_1>0 \textrm{ and that}
\\
&\int_{\R^2_+}
 \frac{1}{x_1}|\nabla u_m|^2\eta\,dx =
-\int_{\R^2_+} u_m \frac{1}{x_1} \nabla u_m\cdot \nabla \eta \,dx
\\&\to
-\int_{\R^2_+} u_0 \frac{1}{x_1}\nabla u_0\cdot \nabla \eta\,dx =
\int_{\R^2_+} \we |\nabla
u_0|^2\eta\,dx
\textrm{ in the case } x^0_1=0
\end{align*} as $m\to \infty$. It therefore follows that $\nabla u_m$ converges strongly in $L^2_w$ (and in $L^2$ if $x^0_1>0$)
to $\nabla u_0$ as
$m\to\infty$.

(iv): Let us take a sequence $r_m\to 0+$ such that $u_m$
defined in (\ref{blo1})-(\ref{blo4}) converges weakly in
$W^{1,2}_{w,\textnormal{loc}}(\R^2_+)$ (weakly in $W^{1,2}_{\textnormal{loc}}(\R^2)$ in the case $x^0_1>0$)
to a function $u_0$.
Using (iii) and the homogeneity of $u_0$,
we obtain that
\begin{align*} &\lim_{m\to \infty} M^{int}(r_m)
 = {1\over {x^0_1}}\left(\int_{B_1} |\nabla u_0|^2\,dx - \int_{\partial B_1} u_0^2
\dh \right)\non \\ \non  &\qquad + \lim_{r\to 0+} r^{-2} \int_{B^+_r(x^0)} x_1x_2 \chi_{\{
u>0\}}\,dx\\
& =  x^0_1 x^0_2 \lim_{r\to 0+} r^{-2} \int_{B^+_r(x^0)} \chi_{\{
u>0\}}\,dx,\\&
\lim_{m\to \infty} M^{x_2}(r_m)
 = {1\over {x^0_1}}\left(\int_{B_1} |\nabla u_0|^2\,dx - \frac{3}{2}\int_{\partial B_1} u_0^2
\dh \right)\non \\ \non  &\qquad + \lim_{r\to 0+} r^{-3} \int_{B^+_r(x^0)} x_1x_2 \chi_{\{
u>0\}}\,dx\\
& =  x^0_1 \lim_{r\to 0+} r^{-3} \int_{B^+_r(x^0)} x_2 \chi_{\{
u>0\}}\,dx,\\&
\lim_{m\to \infty} M^{x_1}(r_m)
 = \int_{B_1^+} \we |\nabla u_0|^2\,dx - 2\int_{\partial B_1^+} \we u_0^2
\dh \non \\ \non  &\qquad + \lim_{r\to 0+} r^{-3} \int_{B^+_r(x^0)} x_1x_2 \chi_{\{
u>0\}}\,dx
\end{align*}
\begin{align*}
& =  x^0_2 \lim_{r\to 0+} r^{-3} \int_{B^+_r(x^0)} x_1\chi_{\{
u>0\}}\,dx,\\&
\lim_{m\to \infty} M^{x_1 x_2}(r_m)
 = \int_{B_1^+} \we |\nabla u_0|^2\,dx - \frac{5}{2}\int_{\partial B_1^+} \we u_0^2
\dh \non \\ \non  &\qquad + \lim_{r\to 0+} r^{-4} \int_{B^+_r(x^0)} x_1x_2 \chi_{\{
u>0\}}\,dx\\
& =  \lim_{r\to 0+} r^{-4} \int_{B^+_r(x^0)} x_1x_2 \chi_{\{
u>0\}}\,dx.\end{align*}
  In the case $x^0_2>0$, $M^{int}(0+)\ge 0$ and $M^{x_1}(0+)\ge 0$, and equality implies that
$u_m$ converges to $0$ in measure in $\R^2_+$.
\end{proof}

The next lemma will be useful in the characterization of blow-up limits in
Proposition \ref{2dim}.

\begin{lemma}\label{legendre}
The Legendre function $y=P_{3/2}$ satisfies
$$x\mapsto\frac{y'(x)}{y'(-x)} \textrm{ is strictly increasing on } (-1,1).$$
\end{lemma}
\begin{proof}
It suffices to prove that
$$y''(x)y'(-x)+y''(-x)y'(x)>0 \textrm{ in } (-1,1).$$
Using the differential equation
$$(1-x^2)y''(x)-2xy'(x)+ \frac{3}{2}\frac{5}{2}y(x)=0,$$
we obtain
$$y''(x)y'(-x)+y''(-x)y'(x)=-\frac{15}{4}\frac{1}{1-x^2}
(y(x)y'(-x)+y(-x)y'(x)).$$
Therefore it is sufficient to prove that $f(x)=y(x)y'(-x)+y(-x)y'(x)<0$
in $(-1,1)$. As $f(x)\to -\infty$ for $|x|\to 1$,
must have a maximum point in $(-1,1)$.
At the maximum point,
$$0=f'(x)=\frac{2x}{1-x^2} (y(x)y'(-x)+y(-x)y'(x)),$$
implying that $x=0$ and that $$\max f = 2 y(0)y'(0)=3\frac{\sqrt{\pi}}{\Gamma(-{1\over 2})\Gamma({7\over 4})}P_{1\over 2}(0)$$ $$= 3\frac{\sqrt{\pi}}{\Gamma(-{1\over 4})\Gamma({7\over 4})}
\frac{\sqrt{\pi}}{\Gamma({1\over 4})\Gamma({5\over 4})}<0$$
(see http://functions.wolfram.com/07.07.20.0006.01,\\
http://functions.wolfram.com/07.07.03.0001.01).
\end{proof}
\begin{proposition}[Characterization of blow-up Limits]\label{2dim}
Let $u$ be a variational solution of {\rm
(\ref{strongp})}, and suppose that
\[\frac{|\nabla u|^2}{x_1}\le C  x_1 |x_2|\quad\text{locally in }\Omega,\]
and that
\[\int_{B^+_r(x^0)} \sqrt{x_2^+}\,d|\nabla \chi_{\{
u>0\}}|\le C_1 \left\{\begin{array}{l} r^{3\over 2}, x^0_2=0\\
r \sqrt{x^0_2}, x^0_2> 0\end{array}\right.\]
for all sufficiently small $r>0$.
\\
Then the following hold:
\\
(i) In the case $x^0_1>0$ and $x^0_2>0$,
the only possible blow-up limits of $u_m(x) = u(x^0+{r_m}x)/r_m$ are
\[u_0(x)= x^0_1\sqrt{x^0_2}\max(x\cdot e,0)\qquad\text{and}\qquad u_0(x)=\gamma |x\cdot e|,\]
where $e$ is a unit vector and $\gamma$ is a nonnegative constant.
In the case $u_0(x)= x^0_1\sqrt{x^0_2}\max(x\cdot e,0)$, the corresponding
density is $M^{int}(0+)=x^0_1x^0_2
\omega_2/2$,
in the case $u_0(x)=\gamma |x\cdot
e|$ with $\gamma>0$ the density is $M^{int}(0+)=x^0_1x^0_2\omega_2$, while in the case
$u_0=0$ the density has possible values  $M^{int}(0+)\in\{0,x^0_1x^0_2\omega_2\}$.
\\
(ii) In the case $x^0_1>0$ and $x^0_2=0$, the only possible blow-up limits are
\[u_0(\rho \sin \theta,\rho \cos \theta)= \frac{\sqrt{2}x^0_1}{3}\rho^{3/2}\cos({3\over 2}\theta)\chi_{\{(\rho\sin\theta,\rho\cos\theta): -\pi/3<\theta<\pi/3\}},\]
 with corresponding density
 \[ M^{x_2}(0+)=x^0_1 \int_{B_1} x_2\chi_{\{(\rho\sin\theta,\rho\cos\theta): -\pi/3<\theta<\pi/3\}}\,dx,\]
 and $u_0(x)=0$, with possible values
of the density \[ M^{x_2}(0+)\in \left\{x^0_1 \int_{B_1} x_2^+\,dx, x^0_1 \int_{B_1} x_2^-\,dx, 0\right\}.\]
\\
(iii) In the case $x^0_1=0$ and $x^0_2>0$, the only possible blow-up limits are
\[u_0(x) = \gamma x_1^2\]
with $\gamma$ a nonnegative constant and
 corresponding density
$$M^{x_1}(0+)=x^0_2 \int_{B^+_1} x_1\,dx,$$
and $u_0(x)=0$, with possible values
of the density \[M^{x_1}(0+)\in\left\{x^0_2 \int_{B^+_1} x_1\,dx, 0\right\}.\]
\\
(iv) In the case $x^0_1=x^0_2=0$, the only possible blow-up limits are
\[u_0(\rho\sin\theta,\rho\cos\theta)= r^{5\over 2}\> U_\ell(\theta)\]
 with corresponding density
\[ M^{x_1x_2}(0+)=\int_{B^+_1\cap \{(\rho\sin\theta,\rho\cos\theta): P'_{3/2}(-\cos \theta)<0\}} x_1 x_2\,dx ,\]
where $P_{3/2}$ is the Legendre function and $U_\ell$ is a unique
function which is positive in $B^+_1\cap \{P'_{3/2}(-\cos \theta)<0\}$
(an angle of $\approx 114.799^\circ$ in the positive $x_2$-direction)
and zero else,
 and $u_0(x)=0$, with possible values
of the density \[ M^{x_1x_2}(0+)\in\left\{\int_{B^+_1} x_1 x_2^+\,dx, \int_{B^+_1} x_1 x_2^-\,dx,0\right\}.\]
For $U_\ell$ we have the relations
$$\frac{5}{2} U_\ell(\theta)=c_0 \sin^2 \theta P_{3/2}'(\cos\theta),
 U_\ell'(\theta) = c_0 \frac{3}{2} \sin\theta P_{3/2}(\cos \theta)$$
with a unique positive constant $c_0$.
\end{proposition}
\begin{proof}
Consider a blow-up sequence $u_m$ as in
Lemma \ref{density_1}, where $r_m\to 0+$, with blow-up limit
$u_0$. Because of the strong convergence of $\nabla u_m$ to $\nabla u_0$ in $L^2$
and the compact embedding from
$BV$ into $L^1$, $u_0$ is a homogeneous solution of
\begin{align}
\label{dmvint} 0 = & \int_{\R^2} \frac{1}{x^0_1}\Big( {\vert \nabla u_0 \vert}^2
\div\phi - 2 \nabla u_0 D\phi \nabla u_0 \Big)\,dx +x^0_1 x^0_2 \int_{\R^2}
\chi_0 \div\phi\,dx\\
\non & \textrm{ in the case } x^0_1>0 \textrm{ and } x^0_2> 0,\\
\label{dmvx2}0
 = &\int_{\R^2} \frac{1}{x^0_1} \Big( {\vert \nabla u_0 \vert}^2 \div\phi -
2 \nabla u_0 D\phi \nabla u_0 \Big)\,dx
 + \int_{\R^2} \Big(x^0_1 x_2 \chi_0 \div\phi + x^0_1\chi_0 \phi_2\Big)\,dx\\
\non &\textrm{ in the case } x^0_1>0 \textrm{ and } x^0_2=0,\\
\label{dmvx1}0=&
\int_{\R^2_+} \frac{1}{x_1} \Big( {\vert \nabla u_0 \vert}^2 \div\phi
-\frac{1}{x_1}\vert\nabla u_0\vert^2 \phi_1
-
2 \nabla u_0 D\phi \nabla u_0 \Big)\,dx\\
\non & + \int_{\R^2_+} \Big(x_1 x^0_2 \chi_0 \div\phi\,dx + x^0_2 \chi_0 \phi_1 \Big)\,dx\\
\non &\textrm{ in the case } x^0_1=0 \textrm{ and }x^0_2> 0;\\
\label{dmvx1x2}0=&
\int_{\R^2_+} \frac{1}{x_1} \Big({\vert \nabla u_0 \vert}^2 \div\phi
-\frac{1}{x_1}\vert\nabla u_0\vert^2 \phi_1
-
2 \nabla u_0 D\phi \nabla u_0 \Big)\,dx\\
\non & + \int_{\R^2_+} \Big(x_1 x_2 \chi_0 \div\phi + x_2 \chi_0 \phi_1 + x_1\chi_0 \phi_2\Big)\,dx\\
\non &\textrm{ in the case } x^0_1=x^0_2=0;
\end{align}
the formulas are valid for every
$\phi=(\phi_1,\phi_2) \in C^1_0(\R^2;\R^2)$ in the case $x^0_1>0$ and
for every
$\phi=(\phi_1,\phi_2) \in C^1_0(\R^2;\R^2)$ such that $\phi_1=0$ on $\{ x_1=0\}$
in the case $x^0_1=0$.
Moreover
 $\chi_0$ is the strong $L^1_{\textnormal{loc}}$-limit of
$\chi_{\{u_m>0\}}$ along a subsequence. The values of the function
$\chi_0$ are almost everywhere in $\{0, 1\}$, and the locally
uniform convergence of $u_m$ to $u_0$ implies that $\chi_0=1$ in $\{
u_0>0\}$. Moreover $\chi_0$ is  constant in each connected component
of $\{ u_0=0\}^\circ\setminus \{ x_2=0\}$.
In the case $u_0= 0$, (\ref{dmvint})-(\ref{dmvx1x2}) show that
$\chi_0$ is constant in $\{ x_2\ne 0\}$ in the cases (\ref{dmvx2}) and (\ref{dmvx1x2})
and that $\chi_0$ is constant in the cases (\ref{dmvint}) and (\ref{dmvx1}).
Its value may be either $0$ or $1$.

 Let $z$ be an arbitrary  point in $
\partial\{ u_0=0\} \setminus \{ 0\}$. Consider first the case when $B_\delta(z)\cap\{u_0>0\}$ has exactly one connected
component.
Note that the normal to
$\partial \{ u_0=0\}$ has the constant value
$\nu(z)$ in $B_\delta(z)$ for some $\delta>0$. Plugging in
$\phi(x):= \eta(x)\nu(z)$ into (\ref{dmvint})-(\ref{dmvx1x2}),
where $\eta \in C^1_0(B_\delta(z))$ is arbitrary, and integrating by
parts, it follows that
\begin{align}
\label{dmvint2}0= & \int_{\partial\{ u_0>0\}} \left(
-\frac{1}{x^0_1}|\nabla u_0|^2 + x^0_1x^0_2 (1-\bar\chi_0)\right)\eta
\,d\mathcal{H}^1\\
\non & \textrm{ in the case } x^0_1>0 \textrm{ and } x^0_2> 0,\\
\label{dmvx22}0= &\int_{\partial\{ u_0>0\}} \left(
-\frac{1}{x^0_1}|\nabla u_0|^2 + x^0_1x_2 (1-\bar\chi_0)\right)\eta
\,d\mathcal{H}^1\\
\non &\textrm{ in the case } x^0_1>0 \textrm{ and } x^0_2=0,\\
\label{dmvx12}0= & \int_{\partial\{ u_0>0\}} \left(
-\frac{1}{x_1}|\nabla u_0|^2 + x_1x^0_2 (1-\bar\chi_0)\right)\eta
\,d\mathcal{H}^1\\
\non & \textrm{ in the case } x^0_1=0 \textrm{ and } x^0_2> 0,\\
\label{dmvx1x22}0= &\int_{\partial\{ u_0>0\}} \left(
-\frac{1}{x_1}|\nabla u_0|^2 + x_1x_2 (1-\bar\chi_0)\right)\eta
\,d\mathcal{H}^1\\
\non &\textrm{ in the case } x^0_1=x^0_2=0.
\end{align}
Here
$\bar\chi_0$ denotes the constant value of $\chi_0$ in $\{
u_0=0\}^\circ$. Note that by Hopf's principle, $\nabla u_0\cdot
\nu\ne 0$ on $B_\delta(z)\cap \partial\{ u_0>0\}$. In all
cases it follows therefore that $\bar\chi_0\neq 1$, and
hence necessarily $\bar\chi_0=0$. We deduce from (\ref{dmvint2})-(\ref{dmvx1x22})
that
\begin{align*}
|\nabla u_0|^2=&(x^0_1)^2x^0_2 \textrm{ on } \partial \{
u_0>0\} \\
\non & \textrm{ in the case } x^0_1>0 \textrm{ and } x^0_2> 0,\\
|\nabla u_0|^2=&(x^0_1)^2x_2 \textrm{ on } \partial \{
u_0>0\} \\
\non & \textrm{ in the case } x^0_1>0 \textrm{ and } x^0_2=0,\\
|\nabla u_0|^2=&x_1^2x^0_2 \textrm{ on } \partial \{
u_0>0\} \\
\non & \textrm{ in the case } x^0_1=0 \textrm{ and } x^0_2> 0,\\
|\nabla u_0|^2=&x_1^2x_2 \textrm{ on } \partial \{
u_0>0\} \\
\non & \textrm{ in the case } x^0_1=x^0_2=0.
\end{align*}

Next, let us try to compute $u_0$:
In the cases where $x^0_1>0$,
the homogeneity of $u_0$ and its harmonicity in
$\{u_0>0\}$ imply the following:
if $x^0_2>0$, then
each connected component of $\{u_0>0\}$ is a
half-plane passing trough the origin.
If $x^0_2=0$, then the fact that $u_0$ must be harmonic
in $\{ x_2<0\}$, implies that $\{u_0>0\}$ is
a
cone with vertex at the origin and of opening angle $120^\circ$
symmetric with respect to and containing $\{ (0, t): t>0\}$.

In the cases where $x^0_1=0$,
solving the resulting ODE leads to hypergeometric functions and
is slightly awkward, so we will instead use, in each section
of the unit disk where $u_0>0$, the velocity potential $\phi$
defined by
\begin{align*}
\partial_1 \phi = {1\over x_1}\partial_2 u, \partial_2 \phi = -{1\over x_1}\partial_1 u.
\end{align*}
In the case $x^0_2>0$ we obtain that $\phi(\rho\sin\theta,\rho\cos\theta)$
is homogeneous of degree $1$ and is on the unit circle given by a linear combination of
$P_1(\cos \theta)$ and $\Re(Q_1(\cos \theta))$, where $P_1$ and $Q_1$ are
the Legendre functions.
Now $P_1(x)=x$ and $\Re Q_1$ is a strictly convex function with singularities
at $-1$ and $1$, so that it is not possible that
$$\alpha P_1'(x)+ (\Re Q_1)'(x)=\alpha P_1'(y)+ (\Re Q_1)'(y) \textrm{ for } x\ne y \in (-1,1).$$
It follows that there can be at most one free surface point
of the solution $\alpha P_1(\cos\theta)+ \Re Q_1(\cos\theta)$ in $(0,\pi)$, but then the solution
would have at least one
singularity in the interval $[0,\pi]$. Thus the only solution possible is
$\sigma  P_1(\cos \theta)= \sigma  \cos \theta$, so that
$\phi(x)=\sigma x_2$ and $u_0(x)=c x_1^2$, where $c$ and $\sigma $ are non-negative constants.
The statement about the density follows as $\chi_0=1$ in $\{ u_0>0\}$.

In the case $x^0_2=0$ we obtain that $\phi(\rho\sin\theta,\rho\cos\theta)$
is homogeneous of degree $3/2$ and is on the unit circle given by a liner combination of
$P_{3/2}(\cos \theta)$ and $P_{3/2}(-\cos \theta)$, where $P_{3/2}$ is the
Legendre function. It is well known that $P_{3/2}$ has only one singularity at $-1$
and that $P'_{3/2}$ has in $(-1,1)$ a unique zero $z_0\in (-1,0)$.
By Lemma \ref{legendre}
we obtain
as in the last case
that $\alpha P_{3/2}(\cos \theta)+\beta P_{3/2}(-\cos \theta)$
can have at most one free surface point in $(0,\pi)$.
but then the solution
would have at least one
singularity in the interval $[0,\pi]$ unless $\beta=0$.
The fact that the singularity and the unique zero are both contained in
$[-1,0)$
implies therefore that either
\begin{align*}
\phi(\rho\sin\theta,\rho\cos\theta)= \sigma\rho^{3/2}  P_{3/2}(\cos \theta) \textrm{ in } \{0< \theta < \arccos(z_0)\}\\
\intertext{or}\\
\phi(\rho\sin\theta,\rho\cos\theta)= \sigma \rho^{3/2} P_{3/2}(-\cos \theta) \textrm{ in } \{\arccos(-z_0)<\theta<\pi\}.
\end{align*}
However the free surface must not intersect $\{x_2<0\}$, so that we obtain that the only admissible solution is
$$\phi(\rho\sin\theta,\rho\cos\theta)= \sigma \rho^{3/2} P_{3/2}(-\cos \theta) \textrm{ in } \{\arccos(-z_0)<\theta<\pi\}$$
for some nonzero constant $\sigma $.
Switching from the velocity potential back to $u_0$ we obtain the statement about
$u_0$ as well as the density.

Last, consider the situation when the set
$B_\delta(z)\cap\{u_0>0\}$ has two connected components.
The computations of $u_0$ in the respective cases show that
this is only possible for $x^0_1>0$ and $x^0_2>0$.
The argument for
(\ref{dmvint2}) yields in this case that the constant values of $|\nabla
u_0|^2$ on either side of $\partial \{u_0>0\}$ are equal. This
concludes the proof.
\end{proof}
\begin{lemma}\label{zero}
Let $u$ be a weak solution of {\rm(\ref{strongp})}
such that $u=0$ in $\{x_2\le 0\}$
and
suppose that
\[\frac{|\nabla u|^2}{x_1} \leq x_1 x_2^+\quad\text{in }\Omega.\]
 Then $x^0_2=0$,
$x^0_1>0$ and $M^{x_2}(0+)=0$
imply that $u\equiv 0$ in some
open $2$-dimensional ball containing $x^0$, while
$x^0_1=x^0_2=M^{x_1x_2}(0+)=0$ implies that $u\equiv 0$ in $B_\delta^+$ for some $\delta>0$.
\end{lemma}

\begin{proof} Suppose towards a contradiction that $x^0\in \partial\{
u>0\}$, and let us take a blow-up sequence \[u_m(x) :=
{u(x^0+{r_m}x)/{r_m^{3/2}}}\] converging weakly in
$W^{1,2}_{\textnormal{loc}}(\R^2)$ to a blow-up limit $u_0$ in the case that $x^0_1>0$, and
a blow-up sequence \[u_m(x) :=
{u(x^0+{r_m}x)/{r_m^{5/2}}}\] converging weakly in
$W^{1,2}_{w,\textnormal{loc}}(\R^2_+)$ to a blow-up limit $u_0$ in the case that $x^0_1=0$.
Proposition \ref{2dim} shows that $u_0=0$ in $\R^2$. Consequently,
\begin{align}\label{meas0}
&0\gets  \div ({1\over {x^0_1+r_m x_1}} \nabla u_m) (B_2)\ge \int_{B_2\cap
\partial_{\textnormal{red}} \{ u_m>0\}} \sqrt{x_2} \,d\mathcal{H}^1
\textrm{ in the case } x^0_1>0,\\ \non
&0\gets  \div ({1\over {x_1}} \nabla u_m) (B_2^+)\ge \int_{B_2^+\cap
\partial_{\textnormal{red}} \{ u_m>0\}} \sqrt{x_2} \,d\mathcal{H}^1
\textrm{ in the case } x^0_1=0
\end{align}
as $m\to\infty$. (Recall that $\div ({1\over {x_1}} \nabla u)$ is
a nonnegative Radon measure in $\Omega$.) On the other hand, there
is at least one connected component $V_m$ of $\{ u_m>0\}$ touching
the origin and containing by the maximum principle a point $x^m\in
\partial A$, where $A = (-1,1)\times (0,1)$ in the case $x^0_1>0$
and $A= (0,1)\times (0,1)$
in the case $x^0_1=0$. If $\max\{x_2: x \in
V_m\cap \partial A\}\not\to 0$ as $m\to\infty$, we immediately
obtain a contradiction to (\ref{meas0}). If $\max\{ x_2: x \in
V_m\cap \partial A\}\to 0$, we use the free-boundary condition as
well as $|\nabla u|^2/x_1^2 \leq x_2^+$ to obtain
\begin{align*}
&0=\div ({1\over {x^0_1+r_mx_1}} \nabla u_m) (V_m\cap
A) \le \int_{V_m\cap \partial A} \sqrt{x_2} \,d\mathcal{H}^1 -
\int_{A\cap
\partial_{\textnormal{red}} V_m} \sqrt{x_2} \,d\mathcal{H}^1\\&\qquad\qquad\qquad \qquad\qquad\qquad\textrm{ in the case } x^0_1>0,\\
&0=\div ({1\over {x_1}} \nabla u_m) (V_m\cap
A) \le \int_{V_m\cap \partial A} \sqrt{x_2} \,d\mathcal{H}^1 -
\int_{A\cap
\partial_{\textnormal{red}} V_m} \sqrt{x_2} \,d\mathcal{H}^1\\&\qquad\qquad\qquad\qquad\qquad\qquad\textrm{ in the case } x^0_1=0.
\end{align*} However
$\int_{V_m\cap \partial A} \sqrt{x_2} \,d\mathcal{H}^1$ is the
unique minimiser of $\int_{\partial D} \sqrt{x_2} \,d\mathcal{H}^1$
with respect to all open sets $D$ with $D=V_m$ on $\partial A$. So
$V_m$ cannot touch the origin, a contradiction.\end{proof}

\begin{theorem}[Curve Case]\label{curve}
Let $u$ be a weak solution of (\ref{strongp}) satisfying
\[\frac{|\nabla u|^2}{x_1}\le C  x_1 |x_2|\quad\text{locally in }\Omega,\]
and let $x^0\in\Om$
be such that $x^0_1x^0_2=0$.
Suppose in addition that $\partial\{ u>0\}\cap B_1^+(x^0)$ is in a neighborhood
of $x^0$ a continuous injective curve $\sigma:I\to \R^2$
such that $\sigma=(\sigma_1,\sigma_2)$ and $\sigma(0)=x^0$. Then the
following hold:

$(i_1)$ {\em Stokes corner:} If $x^0_1> 0$, $x^0_2=0$ and $$M^{x_2}(0+)= x^0_1 \int_{B^+_1} x_2\chi_{\{(\rho\sin\theta, \rho\cos\theta): -\pi/3<\theta<\pi/3\}}\,dx,$$ then (cf. Figure \ref{fig1})
$\sigma_1(t)\ne x^0_1$ in $(-t_1,t_1)\setminus \{ 0\}$ and, depending on the parametrization, either
$$ \lim_{t\to 0+} \frac{\sigma_2(t)}{\sigma_1(t)-x^0_1} = \frac{1}{\sqrt{3}}
\textrm{ and } \lim_{t\to 0-} \frac{\sigma_2(t)}{\sigma_1(t)-x^0_1} = -\frac{1}{\sqrt{3}},$$
or
$$ \lim_{t\to 0+} \frac{\sigma_2(t)}{\sigma_1(t)-x^0_1} = -\frac{1}{\sqrt{3}}
\textrm{ and } \lim_{t\to 0-} \frac{\sigma_2(t)}{\sigma_1(t)-x^0_1} = \frac{1}{\sqrt{3}}.$$

\begin{minipage}{\textwidth}
\begin{center}
\begin{picture}(0,0)%
\includegraphics{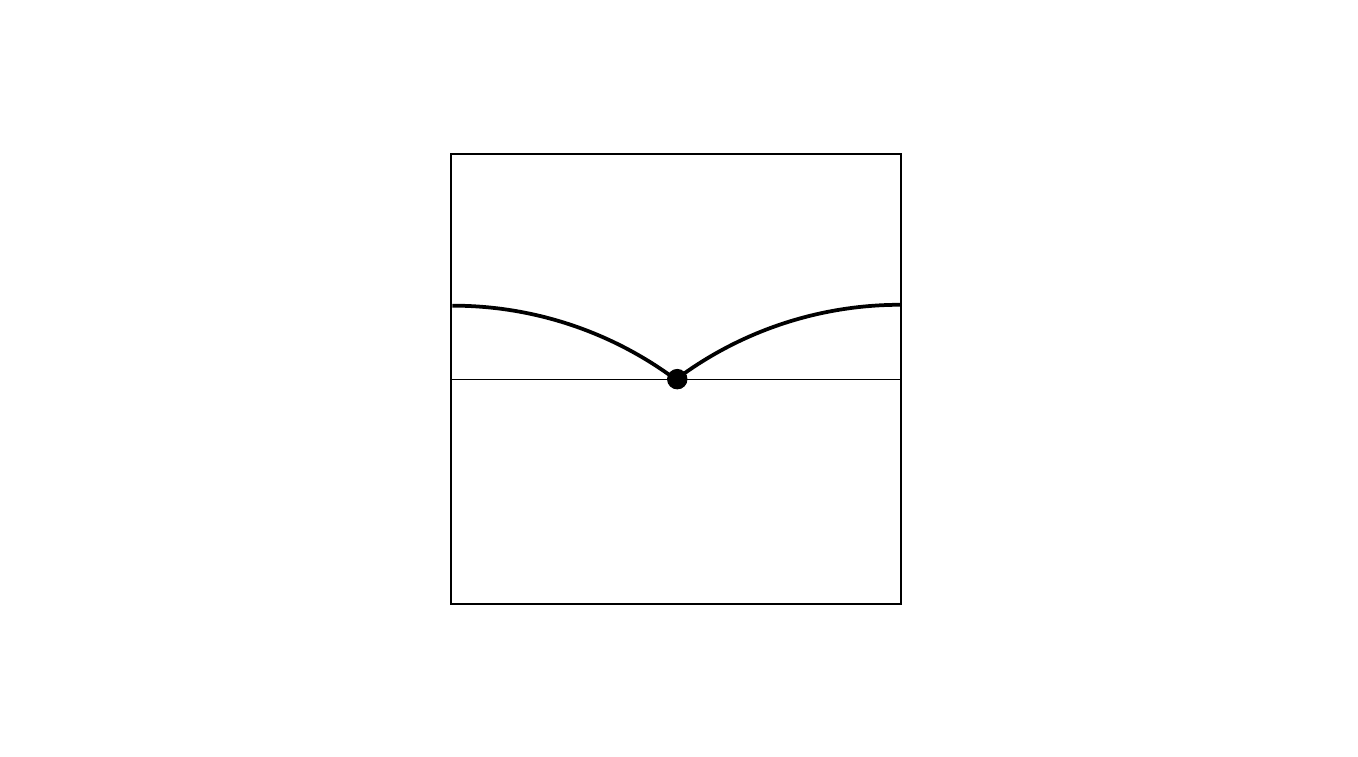}%
\end{picture}%
\setlength{\unitlength}{2368sp}%
\begingroup\makeatletter\ifx\SetFigFont\undefined%
\gdef\SetFigFont#1#2#3#4#5{%
  \reset@font\fontsize{#1}{#2pt}%
  \fontfamily{#3}\fontseries{#4}\fontshape{#5}%
  \selectfont}%
\fi\endgroup%
\begin{picture}(10834,6044)(-2411,-5183)
\put(2594,-1234){\makebox(0,0)[lb]{\smash{{\SetFigFont{11}{13.2}{\rmdefault}{\mddefault}{\updefault}{\color[rgb]{0,0,0}$u>0$}%
}}}}
\put(2567,-3115){\makebox(0,0)[lb]{\smash{{\SetFigFont{11}{13.2}{\rmdefault}{\mddefault}{\updefault}{\color[rgb]{0,0,0}$u=0$}%
}}}}
\end{picture}%

\end{center}
\captionof{figure}{Stokes corner ($x^0_1>0, x^0_2=0$)}\label{fig1}
\end{minipage}

$(i_2)$ If $x^0_1> 0$, $x^0_2=0$ and $M^{x_2}(0+)=x^0_1\int_{B^+_1} x^+_2\,dx$
or $M^{x_2}(0+)=x^0_1\int_{B^+_1} x^-_2\,dx$,
then (cf. Figure \ref{fig6})
$\sigma_1(t)\ne x^0_1$ in $(-t_1,t_1)\setminus \{ 0\}$,
$\sigma_1-x^0_1$ changes sign at $t=0$ and
$$ \lim_{t\to 0} \frac{\sigma_2(t)}{\sigma_1(t)-x^0_1} = 0.$$

\begin{minipage}{\textwidth}
\begin{center}
\begin{picture}(0,0)%
\includegraphics{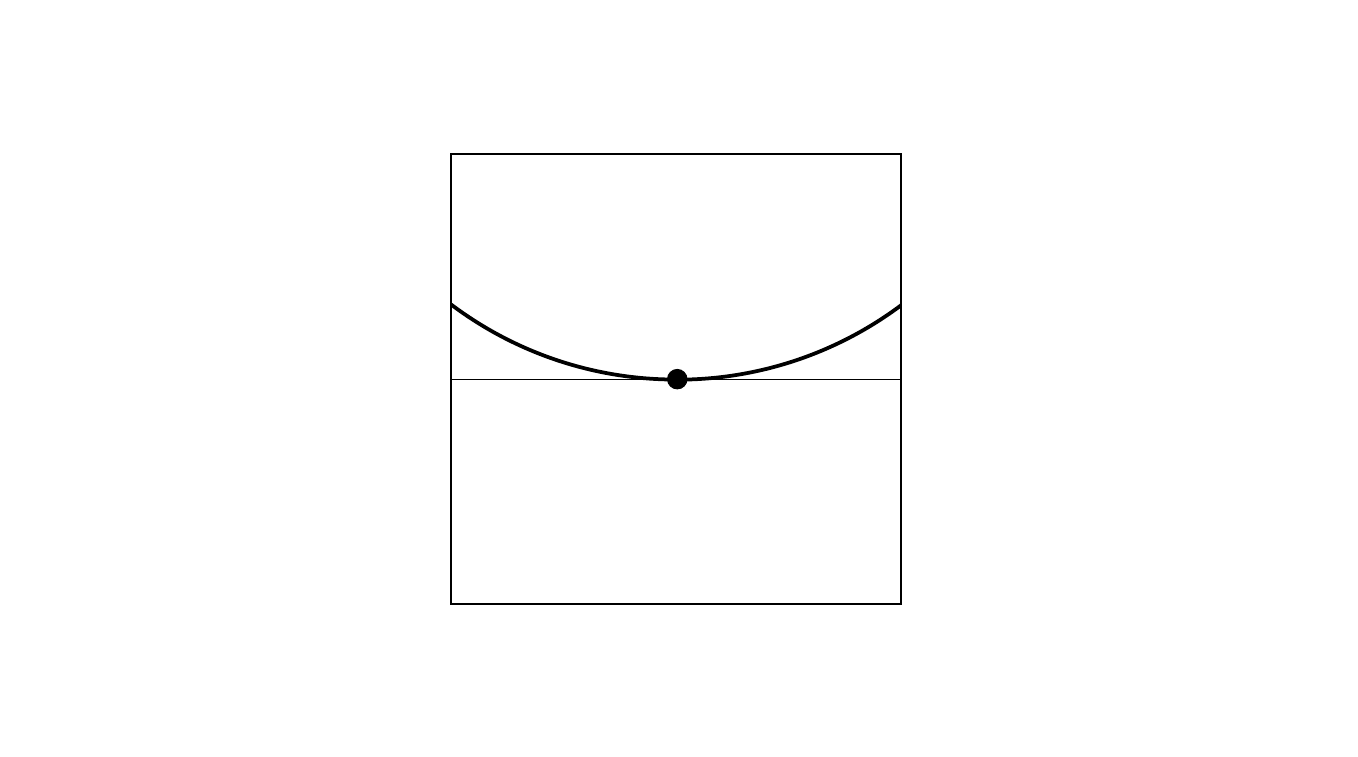}%
\end{picture}%
\setlength{\unitlength}{2368sp}%
\begingroup\makeatletter\ifx\SetFigFont\undefined%
\gdef\SetFigFont#1#2#3#4#5{%
  \reset@font\fontsize{#1}{#2pt}%
  \fontfamily{#3}\fontseries{#4}\fontshape{#5}%
  \selectfont}%
\fi\endgroup%
\begin{picture}(10834,6044)(-2411,-5183)
\put(2594,-1234){\makebox(0,0)[lb]{\smash{{\SetFigFont{11}{13.2}{\rmdefault}{\mddefault}{\updefault}{\color[rgb]{0,0,0}$u>0$}%
}}}}
\put(2567,-3115){\makebox(0,0)[lb]{\smash{{\SetFigFont{11}{13.2}{\rmdefault}{\mddefault}{\updefault}{\color[rgb]{0,0,0}$u=0$}%
}}}}
\end{picture}%

\end{center}
\captionof{figure}{Horizontal point ($x^0_1>0, x^0_2=0$)}\label{fig6}
\end{minipage}

$(i_3)$ In the case $x^0_1>0$, $x^0_2=0$ and $M^{x_2}(0+)=0$ ---which is according to Lemma \ref{zero}
not possible at all provided that $u=0$ in $\{x_2\le 0\}$ and the sharp Bernstein inequality holds---,
then
$\sigma_1(t)\ne x^0_1$ in $(-t_1,t_1)\setminus \{ 0\}$,
 $\sigma_1-x^0_1$ does not change its sign at $t=0$, and
$$ \lim_{t\to 0} \frac{\sigma_2(t)}{\sigma_1(t)-x^0_1} = 0.$$

$(ii_1)$ If $x^0_2>0$, $x^0_1=0$ and $M^{x_1}(0+)=x^0_2\int_{B^+_1} x_1\,dx$, then (cf. Figures \ref{fig3}-\ref{fig5})
$\sigma_2(t)\ne x^0_2$ in $(0,t_1)$ and
$$ \lim_{t\to 0} \frac{\sigma_1(t)}{\sigma_2(t)-x^0_2} = 0,$$

\begin{figure}
\begin{center}
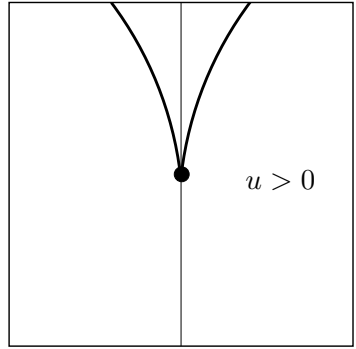
\end{center}
\captionof{figure}{Downward Vertical Cusp ($x^0_1=0, x^0_2>0$)}\label{fig3}
\end{figure}
\begin{figure}
\begin{center}
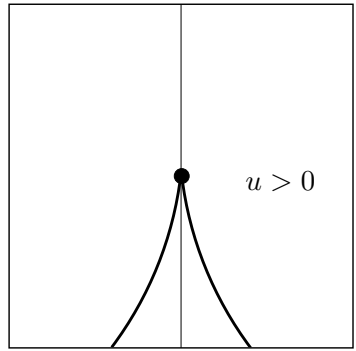
\end{center}
\captionof{figure}{Upward Vertical Cusp ($x^0_1=0, x^0_2>0$)}\label{fig4}
\end{figure}
or
$\sigma_2(t)\ne x^0_1$ in $(-t_1,t_1)\setminus \{ 0\}$,
$\sigma_2-x^0_2$ changes sign at $t=0$ and
$$ \lim_{t\to 0} \frac{\sigma_1(t)}{\sigma_2(t)-x^0_2} = 0.$$

\begin{figure}
\begin{center}
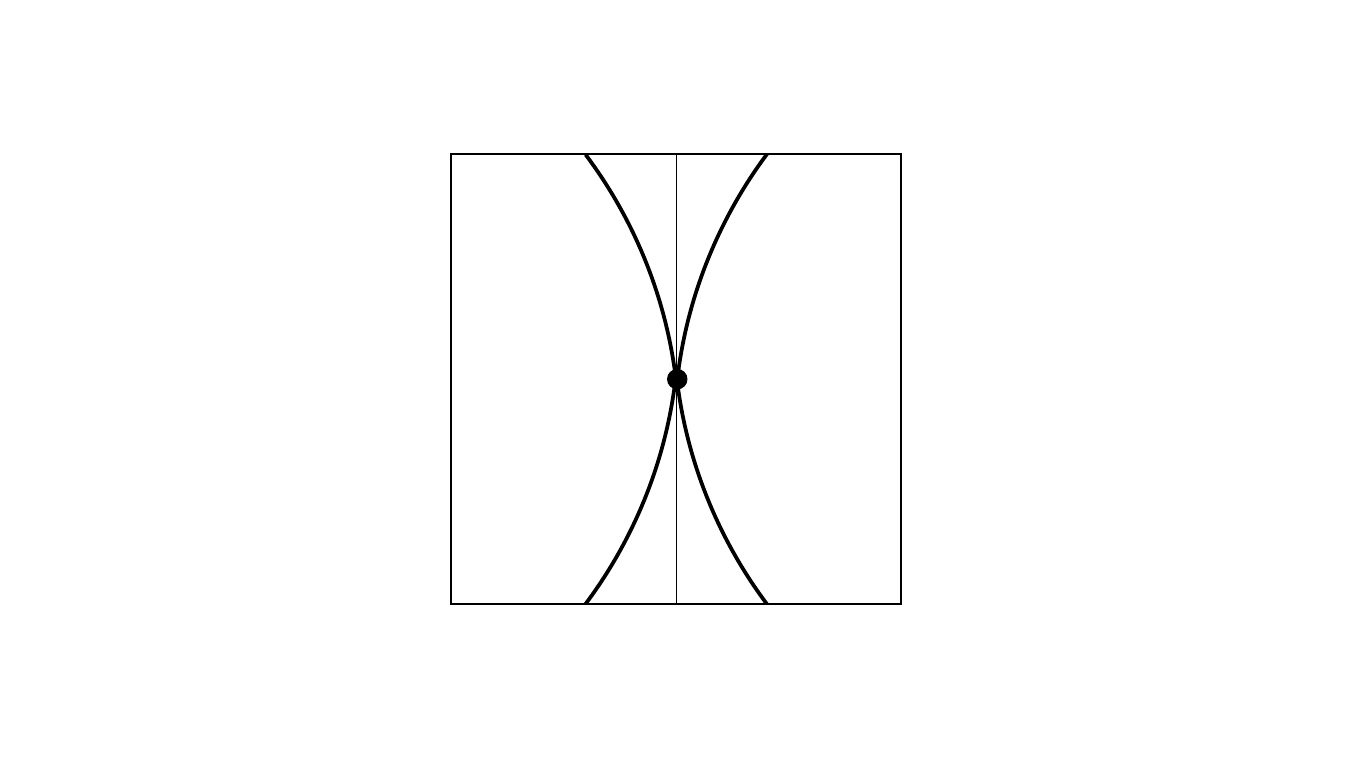
\end{center}
\captionof{figure}{Double Vertical Cusp ($x^0_1=0, x^0_2>0$)}\label{fig5}
\end{figure}

$(ii_2)$ The case $x^0_2>0$, $x^0_1=0$ and $M^{x_1}(0+)=0$
is not possible.

$(iii_1)$ {\em Garabedian corner:} If $x^0_1=x^0_2=0$ and $$M^{x_1x_2}(0+)= \int_{B^+_1\cap \{P'_{3/2}(-\cos \theta)<0\}} x_1 x_2\,dx,$$ then (cf. Figure \ref{fig2})
$\sigma_1(t)\ne 0$ in $(0,t_1)$ and,
$$ \lim_{t\to 0+} \frac{\sigma_2(t)}{\sigma_1(t)} = \tan(\pi/2-\arccos(-z_0)).$$

\begin{figure}
\begin{center}
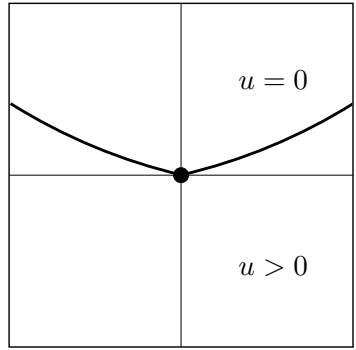
\end{center}
\captionof{figure}{Garabedian corner ($x^0_1=x^0_2=0$)}\label{fig2}
\end{figure}

$(iii_2)$ If $x^0_1=x^0_2=0$ and $$M^{x_1x_2}(0+)=\int_{B^+_1} x_1x^+_2\,dx
\textrm{ or }M^{x_1x_2}(0+)=\int_{B^+_1} x_1x^-_2\,dx,$$
then (cf. Figure \ref{fig7}) $\sigma_1(t)\ne 0$ in $(0,t_1)$ and
$$ \lim_{t\to 0} \frac{\sigma_2(t)}{\sigma_1(t)} = 0.$$
(In the subsequent sections of the present paper we will analyze
the precise asymptotics of the velocity field in the case
$M^{x_1x_2}(0+)=\int_{B^+_1} x_1x^+_2\,dx$.)

\begin{figure}
\begin{center}
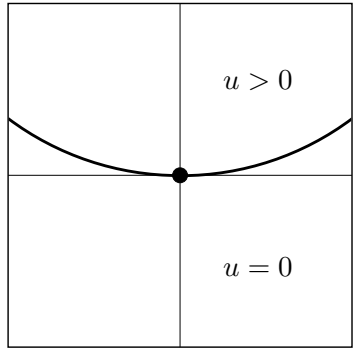
\end{center}
\captionof{figure}{Horizontal point ($x^0_1=x^0_2=0$)}\label{fig7}
\end{figure}

$(iii_3)$ If $x^0_1=x^0_2=0$ and $M^{x_1x_2}(0+)=0$ ---which is according to Lemma \ref{zero}
not possible at all provided that $u=0$ in $\{x_2\le 0\}$ and the sharp Bernstein inequality holds---, then
$\sigma_1(t)\ne 0$ in $(-t_1,t_1)\setminus \{ 0\}$,
 and
$$ \lim_{t\to 0} \frac{\sigma_2(t)}{\sigma_1(t)} = 0.$$
\end{theorem}
\begin{remark}
Although we omit a proof in the present paper, a perturbation of the frequency formula
in \cite{VW} (see \cite{VW2}) can be used to prove that, if $x_1^0>$ and $x_2^0=0$, then case $M^{x_2}(0+)=x^0_1\int_{B^+_1} x^+_2\,dx$ is not possible.
Case $(ii_1)$ seems possible as we have a nontrivial homogeneous solution.
We do at present not have an existence proof for the cusps suggested here.
\end{remark}
\begin{proof}
We prove the claimed results only case (iii), when $x^0_1=x^0_2=0$, the analysis in the other cases being similar.
For each $y=(y_1,y_2)\in \R^2$ with $y_1\geq 0$ and $(y_1,y_2)\neq (0,0)$, we define $\theta(y)\in [0,\pi]$ by the relation
\[(y_1,y_2)=(\rho(y)\sin\theta(y), \rho(y)\cos\theta(y)).\]
 We now consider the set
$${\mathcal L}= \{ \theta_0\in [0,\pi] \> : \>
\textrm{ there is } t_m \to 0\textrm{ such that }\theta(\sigma(t_m)) \to \theta_0
\textrm{ as } m\to\infty\}.$$
Note that in fact ${\mathcal L}\subset[0,\pi/2]$, since the free boundary $\partial\{u>0\}$ is contained in $\{x_2\geq 0\}$.

We now claim that:
{\em The set ${\mathcal L}$ is a subset of $\{ 0, \theta^*,\pi/2\}$, where $\theta^*=\arccos(-z_0)$ is the angle corresponding to the Garabedian cone.}\\
Indeed, suppose towards a contradiction that a sequence $0\ne t_m\to 0+, m\to\infty$ exists
such that $\theta(\sigma(t_m)) \to \theta_0\in {\mathcal L}
\setminus \{ 0, \theta^*,\pi/2\}$, let $r_m := |\sigma(t_m)|$
and let $$u_m(x) := \frac{u({r_m}x)}{r_m^{5/2}}.$$
For each $\rho>0$ such that $\tilde B := B_\rho(\sin\theta_0,\cos \theta_0)$ satisfies
$$\emptyset = \tilde B \cap
\left(\{ (\alpha,0)\> : \alpha\in \R_+\}\cup \{ (0,\alpha)\> : \alpha\in \R_+\} \cup \{ (\alpha\sin\theta^*,\alpha\cos\theta^*)\> : \alpha\in\R_+\}\right),$$
we infer from the formula for the unique blow-up limit $u_0$ (see Theorem \ref{2dim})
that the convergence of measures
$$(\div (\we \nabla  u_m))(\tilde B)\to (\div (\we \nabla  u_0))(\tilde B)=0\textrm{ as } m\to\infty.$$
On the other hand,
$$\div (\we \nabla  u_m)=  \sqrt{x_2} {\mathcal H}^1
\lfloor \partial{\{ u_m>0\}},$$
which implies, since
$\tilde B\cap \partial{\{ u_m>0\}}$ contains a curve of length at least $2\rho-o(1)$,
that
$$0 \gets (\div (\we \nabla  u_m))(\tilde B) \ge c(\theta_0,\rho)\textrm{ as } m\to\infty,$$
where $c(\theta_0,\rho)>0$, a contradiction.
This proves the property claimed.

 Now, a continuity argument yields that ${\mathcal L}$ is a connected set.
Consequently the limit
$$\ell = \lim_{t\to 0+} \theta(\sigma(t))$$ exists and is contained in the set
$\{ 0, \theta^*,\pi/2\}$. In what follows, we identify the value of $\ell$ in terms of the value of $M^{x_1x_2}(0+)$.

Suppose first that $M^{x_1x_2}(0+)= \int_{B^+_1\cap \{P'_{3/2}(-\cos \theta)<0\}} x_1 x_2\,dx$. Then, by Proposition \ref{2dim}, the blow-up limit is
\[u_0(\rho\sin\theta,\rho\cos\theta)= r^{5\over 2}\> U_\ell(\theta).\]
Since $ (\div (\we \nabla  u_0))(B_{1/100}(\sin\theta^*,\cos\theta^*))>0$, it follows that we cannot have $\ell\in\{0,\pi/2\}$, and therefore we must have
$\ell=\theta^*$. This proves case $(iii_1)$ of the Theorem.

Suppose now that $M^{x_1x_2}(0+)\in \left\{\int_{B^+_1} x_1 x_2^+\,dx, \int_{B^+_1} x_1 x_2^-\,dx,0\right\}$. Then the blow-up limit is $u_0(x)=0$. The same argument given earlier in the proof shows that $\ell\neq\theta^*$, so that necessarily $\ell\in\{0,\pi/2\}$. But then the formula in Lemma \ref{density_1} that
$$M^{x_1x_2}(0+)=\lim_{r\to 0+} r^{-4}
\int_{B^+_r(x^0)}x_1 x_2 \chi_{\{ u>0\}}\,dx,$$
shows that $\ell=0$ implies $M^{x_1x_2}(0+)=0$, while $\ell=\pi/2$ implies that $M^{x_1x_2}(0+)\in \left\{\int_{B^+_1} x_1 x_2^+\,dx, \int_{B^+_1} x_1 x_2^-\,dx, 0\right\}$. However, the possibility that $M^{x_1x_2}(0+)=0$ and $\ell=0$ is ruled out by the argument in the proof of Lemma \ref{zero}, even in the absence of the strict Bernstein condition. This proves the cases $(iii_2)$ and $(iii_3)$ of the Theorem.

\end{proof}

\section{Frequency formula}
From now on we will focus on the case $x^0_1=x^0_2=0$, $u=0$ in $\{ x_2\le 0\}$ and $M^{x_1x_2}(0+)=\int_{B^+_1} x_1x^+_2\,dx$,
in which we will derive a precise asymptotic profile of the velocity.
\begin{theorem}[Frequency Formula]\label{freq2}
Let $u$ be a variational solution of (\ref{strongp}), and let $\delta:=\dist(0,\partial\Omega)/2$.
Let, for any $r\in (0,\delta)$, \[D(r)= \frac{r\int_{B^+_r(0)}\we|\nabla u|^2\,dx}{\int_{\partial
B^+_r(0)}\we u^2\dh}\]
and
\[V(r)= \frac{r\int_{B^+_r(0)}
x_1x_2(1-\chi_{\{ u>0\}})\,dx}{\int_{\partial B^+_r(0)}\we u^2\dh}.\]
Then the
``frequency''
\begin{align*}H(r)&=D(r)-V(r)\\ &=  \frac{r \int_{B^+_r(0)} \Big(\we |\nabla u|^2+x_1x_2
(\chi_{\{ u>0\}}-1)\Big)\,dx}{\int_{\partial B^+_r(0)}\we u^2\dh}
\end{align*}
satisfies for a.e. $r\in (0,\delta)$ the identities
\begin{align} & H'(r)\non\\
  &= \frac{2}{r} \int_{\partial B^+_r(0)}\we \left[ \frac{r (\nabla
 u\cdot\nu)}{\left(\int_{\partial B^+_r(0)}\we u^2\dh\right)^{1/2}}
 -D(r)\frac{u}{\left(\int_{\partial
B^+_r(0)}\we u^2\dh\right)^{1/2}}\right]^2\dh\non\\&\qquad\qquad
+\frac{2}{r} V^2(r)+\frac{2}{r} V(r)\left(H(r)-\frac{5}{2}\right)\label{ff}\end{align}
and
 \begin{align} & H'(r)\non\\
  &= \frac{2}{r} \int_{\partial B^+_r(0)}\we\left[ \frac{r (\nabla
 u\cdot\nu)}{\left(\int_{\partial B^+_r(0)}\we u^2\dh\right)^{1/2}}
 -H(r)\frac{u}{\left(\int_{\partial
B^+_r(0)}\we u^2\dh\right)^{1/2}}\right]^2\dh\non\\&\qquad\qquad+
\frac{2}{r} V(r)\left(H(r)-\frac{5}{2}\right).\label{newff}\end{align}
\end{theorem}

\begin{proof} Note that, for all $r\in (0,\delta)$,
\be H(r)= \frac{r^{-4}I(r)-\int_{B^+_1} x_1x_2\,dx}{r^{-5}J(r)}.\label{fgs}\ee
Hence, for a.e. $r\in (0,\delta)$,
\[ H'(r)=
\frac{(r^{-4}I(r))'}{r^{-5}J(r)}-\frac{(r^{-4}I(r)-\int_{B^+_1} x_1x_2\,dx)}{r^{-5}J(r)}\frac{(r^{-5}J(r))'}{r^{-5}J(r)},\non\]
Using the identities (\ref{Iprx1x2}) and (\ref{easy}) with $\alpha=-5$, we therefore obtain that, for a.e. $r\in (0,\delta)$,
\begin{align}
H'(r)&=\frac{\left(2r\ipbrx\we(\nabla u\cdot \nu)^2\dh- 5\ipbrx\we u (\nabla u\cdot \nu)\dh \right)}{\ipbrx \we u^2\dh}\non\\
&\qquad -(D(r)-V(r))\frac{1}{r}\frac{\left(2r\ipbrx\we u (\nabla u\cdot \nu)\dh-5\ipbrx \we u^2\dh\right)}{\ipbrx \we u^2\dh}\non\\
&=\frac{2}{r}\left(\frac{r^2\ipbrx\we (\nabla u\cdot \nu)^2\dh}{\ipbrx\we u^2\dh}-\frac{5}{2}D(r)\right)\non\\
&\qquad-\frac{2}{r}(D(r)-V(r))\left(D(r)-\frac{5}{2}\right),\label{wq}
\end{align}
where we have also used the fact, which follows from (\ref{part2}), that
\be D(r)=\frac{r\ipbrx \we u (\nabla u\cdot \nu)\dh}{\ipbrx \we u^2\dh}\label{hfvv}.\ee
Identity (\ref{ff}) now follows by merely rearranging (\ref{wq}), making use again of (\ref{hfvv}) and the fact that $D(r)=V(r)+H(r)$.

Since (\ref{ff}) holds, it follows by inspection that (\ref{newff}) holds if and only if
\begin{align} \int_{\partial B^+_r(0)}&\we\left[{r (\nabla
u\cdot\nu)}
-D(r){u}\right]^2\dh + V^2(r)\int_{\partial B^+_r(0)}\we u^2\dh\non\\& =\int_{\partial B^+_r(0)}\we \left[{r (\nabla
u\cdot\nu)}
-H(r){u}\right]^2\dh\label{ttr}
.\end{align}
However, (\ref{ttr}) is easily verified as a consequence of (\ref{hfvv}) and the fact that $D(r)=H(r)+V(r)$.
In conclusion, identity (\ref{newff}) also holds.
\end{proof}

\begin{theorem}\label{mouthm}
Let $u$ be a variational solution of {\r
(\ref{strongp})} such that $u=0$ in $\{ x_2\le 0\}$, let $x^0=0$, suppose that $M^{x_1x_2}(0+)=\int_{B^+_1} x_1x^+_2\,dx$, and let $\delta:=\dist(0,\partial\Omega)/2$. Then the following hold:

(i) $H(r)\geq\frac{5}{2}\quad\text{for all }r\in (0,\delta)$.

(ii) The function $r\mapsto r^{-5}J(r)$ is nondecreasing on $(0,\delta)$.

(iii) The function $H$ is nondecreasing on $(0,\delta)$, and has a right limit $H(0+)$, where $H(0+)\ge 5/2$.

(iv) $r\mapsto \frac{1}{r}V^2(r)\in L^1(0, \delta)$.

\end{theorem}

\begin{proof}

(i) The monotonicity, which follows from Theorem \ref{elmon2}, of the function $M^{x_1x_2}$ ensures that, for all $r\in (0,\delta)$,
\be r^{-4}I(r)-\frac{5}{2}r^{-5}J(r)\geq \int_{B^+_1} x_1x^+_2\,dx.\label{lala}\ee
Using (\ref{fgs}), the above inequality may be rearranged in the form of the claimed result.

(ii) Plugging  $\alpha=-5$ into (\ref{easy}), using also (\ref{part2}), and then (\ref{lala}), we obtain, for a.e. $r\in (0,\delta)$,
\begin{align*}(r^{-5}J(r))'&=\frac{2}{r}\left( r^{-4}\int_{B^+_r(0)} \we {\vert \nabla u \vert}^2\,dx-\frac{5}{2}r^{-5}\int_{\partial B^+_r(0)}\we u^2\dh\right)\\
&\geq 2r^{-5}\int_{B^+_r(0)}
x_1x_2(1-\chi_{\{ u>0\}})\,dx\geq 0,
\end{align*}
which implies the claimed result.

(iii) The monotonicity of $H$ on $(0,\delta)$ is a consequence of (\ref{ff}) and (i). The remaining part of the claim is immediate.

(iv) The claimed result follows from (\ref{ff}) and (iii).

\end{proof}

\section{Blow-up limits}

The Frequency Formula allows passing to blow-up limits.
\begin{proposition}\label{blowup} Let $u$ be a variational solution of {\rm
(\ref{strongp})} such that $u=0$ in $\{x_2\le 0\}$, let $x^0=0$, and suppose that $M^{x_1x_2}(0+)=\int_{B^+_1} x_1x^+_2\,dx$. Then:

 (i) There exist $\lim_{r\to 0+}V(r)= 0$ and $\lim_{r\to 0+}D(r)= H(0+)$.

(ii) For any sequence $r_m\to 0+$ as $m\to\infty$, the sequence
\be v_m(x) := \frac{u(r_m x)}{\sqrt{r_m^{-1}\int_{\partial
B_{r_m}^+} \we u^2 \dh}}\label{vm}\ee is bounded in $W^{1,2}_w(B_1^+)$.

 (iii) For any sequence $r_m\to 0+$ as $m\to\infty$ such that the sequence $v_m$ in {\rm (\ref{vm})} converges
 weakly in $W^{1,2}_w(B_1^+)$ to a blow-up limit $v_0$, the function $v_0$ is homogeneous of
degree $H(0+)$ in $B_1^+$,  and satisfies
 \[\text{$v_0\geq 0$ in $B_1$,
$v_0\equiv 0$ in $B_1^+\cap\{x_2\leq 0\}$ and  $\int_{\partial
B_1^+}\we v_0^2\dh=1$.}\]
\end{proposition}

\begin{proof} We first prove that, for any sequence $r_m\to 0+$, the sequence $v_m$ defined in (\ref{vm}) satisfies, for every
$0<\tau<\sigma<1$,
\be\int_{B_\sigma^+\setminus
B_\tau^+} \we|x|^{-5}\left[\nabla v_m(x) \cdot
x-H(0+)v_m(x)\right]^2\,dx\to 0\quad\text{as
}m\to\infty.\label{vol}\ee
Indeed, for any such $\tau$ and $\sigma$, it follows by scaling from (\ref{newff}) that,
for every $m$ such that $r_m<\delta$,
\begin{align}\int_\tau^\sigma\frac{2}{r} &\int_{\partial B_r^+}\we\left[
\frac{r (\nabla v_m \cdot\nu)}{\left(\int_{\partial
B_r^+}\we v_m^2\dh\right)^{1/2}}-H(r_mr)\frac{v_m}{\left(\int_{\partial
B_r^+}\we v_m^2\dh\right)^{1/2}}\right]^2\dh\,dr\non\\&\leq
H(r_m\sigma)-H(r_m \tau)\to 0\quad\text{as
}m\to\infty,\non\end{align} as a consequence of Theorem \ref{mouthm} (iii). The above implies that
\begin{align}\int_\tau^\sigma\frac{2}{r} &\int_{\partial
B_r^+}\we\left[ \frac{r (\nabla v_m \cdot\nu)}{\left(\int_{\partial
B_r^+}\we v_m^2\dh\right)^{1/2}}-H(0+)\frac{v_m}{\left(\int_{\partial
B_r^+}\we v_m^2\dh\right)^{1/2}}\right]^2\dh\,dr\non\\&\to 0 \quad\text{as
}m\to\infty.\label{coo}\end{align}
Now note that, for every $r\in
(\tau,\sigma)\subset (0,1)$ and all $m$ as before, it follows by
using Theorem \ref{mouthm} (ii), that
\[\int_{\partial
B_r^+}\we v_m^2\dh=\frac{\int_{\partial B_{r_m r}^+}
\we u^2\dh}{\int_{\partial B_{r_m}^+}\we u^2\dh}\leq r^{5}\leq 1.\]
Therefore (\ref{vol}) follows from (\ref{coo}), which proves our claim. Let us also recall (\ref{hfvv}).

We can now prove all parts of the Proposition.

(i) Suppose towards a contradiction that (i) is not true. Let
$s_m\to 0$ be such that the sequence $V(s_m)$ is bounded away from
$0$.
It is a consequence of Theorem \ref{mouthm}(iv) that
\[ \min_{r\in [s_m, 2s_m]} V(r)\to 0\quad\text{as }m\to\infty.\]
Let $t_m\in [s_m, 2s_m]$ be such that $V(t_m)\to 0$ as $m\to\infty$.
For the choice $r_m:=t_m$ for every $m$, the sequence $v_m$ given by
(\ref{vm}) satisfies (\ref{vol}). The fact that $V(r_m)\to 0$
implies that $D(r_m)$ is bounded, and hence that $v_m$ is bounded in
$W^{1,2}_w(B_1^+)$. Let $v_0$ be any weak limit of $v_m$ along a
subsequence. Note that by the compact embedding $W^{1,2}_w(B_1^+)\hookrightarrow L^2(\partial B_1^+)$,
 $v_0$ has norm $1$ on $L^2_w(\partial B_1^+)$,
since this is true for $v_m$ for all $m$. It follows from
(\ref{vol}) that $v_0$ is homogeneous of degree $H(0+)$.
Note that, by using Theorem \ref{mouthm} (ii),
\begin{align}
V(s_m)&=\frac{s_m^{-4}\int_{B_{s_m}^+}x_1x_2(1-\chi_{\{u>0\}})\,dx}
{s_m^{-5}\int_{\partial B_{s_m}^+}\we u^2\dh}\non\\&\leq
\frac{s_m^{-4}\int_{B_{r_m}^+}x_1x_2(1-\chi_{\{u>0\}})\,dx}
{(r_m/2)^{-5}\int_{\partial B_{r_m/2}^+}\we u^2\dh}\non\\
&\leq\frac{1}{2}\frac{\int_{\partial
B_{r_m}^+}\we u^2\dh}{\int_{\partial
B_{r_m/2}^+}\we u^2\dh}V(r_m)\non\\ &=\frac{1}{2\int_{\partial
B_{1/2}^+}\we v_m^2\dh}V(r_m).\label{smrm}\end{align}
 Since, at least
along a subsequence,
\[\int_{\partial B_{1/2}^+}\we v_m^2\dh\to\int_{\partial B_{1/2}^+}\we v_0^2\dh>0,\]
(\ref{smrm}) leads to a contradiction. It follows that indeed
$V(r)\to 0$ as $r\to 0+$. This implies  that $D(r)\to
H(0+)$.

(ii) Let $r_m$ be an arbitrary sequence with $r_m\to 0+$. The
boundedness of the sequence $v_m$ in $W^{1,2}_w(B_1)$ is equivalent to
the boundedness of $D(r_m)$, which is true by (i).

(iii) Let $r_m\to 0+$ be an arbitrary sequence such that $v_m$
converges weakly to $v_0$. The homogeneity degree $H(0+)$ of
$v_0$ follows directly from (\ref{vol}).
 The fact that $\int_{\partial B_1^+}\we v_0^2\dh=1$ is a consequence of
 $\int_{\partial B_1^+}\we v_m^2\dh=1$ for all $m$,
 and the remaining claims of the Proposition are obvious. The homogeneity of $v_0$,
together with the fact that $v_0$ belongs to $W^{1,2}_w(B_1^+)$, imply (in two dimensions)
that $v_0$ is continuous.

\end{proof}

\section{Concentration compactness}
In the present section we will
prove a concentration compactness result which
allows us to preserve variational solutions in the blow-up limit at
degenerate points and excludes concentration. In order to do so we
combine the concentration compactness result of J.-M. Delort
\cite{delort} with information gained by our Frequency
Formula. In addition, we obtain strong convergence of our blow-up
sequence which is necessary in order to prove our main theorems.

\begin{theorem}\label{comp}
Let $u$ be a variational solution of {\rm
(\ref{strongp})} such that $u=0$ in $x_2\le 0$ and $M^{x_1x_2}(0+)=\int_{B^+_1} x_1x^+_2\,dx$.
Let $r_m\to 0+$
be such that the sequence $v_m$ given by {\rm(\ref{vm})} converges
weakly to $v_0$ in $W_w^{1,2}(B^+_1)$. Then $v_m$ converges to $v_0$
strongly in $W^{1,2}_{w,\textnormal{loc}}(B^+_1\setminus \{ 0\})$, $v_0$ is continuous on $B^+_1$ and
$\div (\frac{1}{x_1} \nabla v_0)$ is a nonnegative Radon measure
satisfying $v_0 \div (\frac{1}{x_1} \nabla v_0)=0$ in the sense of Radon measures in
$B^+_1$.
\end{theorem}

\begin{proof}
Note first that
the homogeneity of $v_0$ given by Proposition \ref{blowup},
together with the fact that $v_0$ belongs to $W^{1,2}_w(B^+_1)$, imply
that $v_0$ is continuous.

Let $\sigma$ and $\tau$ with $0<\tau<\sigma<1$ be arbitrary.
We know that $\div (\frac{1}{x_1} \nabla v_m)\ge 0$ and $\div (\frac{1}{x_1} \nabla v_m)(B^+_{(\sigma+1)/2})\leq C_1$ for all $m$.
We regularize
each $v_m$ to \[\tilde v_m := v_m*\phi_m \in C^\infty(B^+_1),\] where
$\phi_m$ is a standard mollifier such that \[ \div (\frac{1}{x_1} \nabla \tilde v_m)\ge
0, \int_{B^+_\sigma} \div (\frac{1}{x_1} \nabla \tilde v_m)\le C_2<+\infty \quad\text{for
all } m,\] and \[\left\Vert \frac{1}{x_1}(\nabla v_m-\nabla \tilde v_m)\right\Vert_{L^2(B^+_\sigma)}
+ \left\Vert v_m-\tilde v_m\right\Vert_{L^2(B^+_\sigma)} \to
0\quad\text{as }m\to\infty.\]
Let us now consider
the velocity field in three dimensions
$$V^m(X,Y,Z):= \left(-{1\over x_1} \partial_2  \tilde v_m \cos \vartheta ,
-{1\over x_1}  \partial_2  \tilde v_m \sin \vartheta,
{1\over x_1} \partial_1  \tilde v_m\right),$$
where $(X,Y,Z)=(x_1\cos\vartheta, x_1\sin\vartheta, x_2)$, as well as their weak limit
$$V(X,Y,Z):= \left(-{1\over x_1} \partial_2  \tilde v \cos \vartheta,
-{1\over x_1} \partial_2  \tilde v \sin \vartheta ,
{1\over x_1} \partial_1  \tilde v\right).$$
We have that
$V^m$
is divergence free
and satisfies $$\curl V^m = \omega^m = (-\sin \vartheta,\cos \vartheta,0) \alpha^m
\textrm{ in }B_2(0)$$
with a non-negative function $\alpha^m$ that is bounded in $L^1(B_\sigma)$.
It follows that
\begin{align*}
V^m_1& = \Delta_{m1}^{-1} \partial_Z  \omega^m_2\\
V^m_2 &= -\Delta_{m2}^{-1} \partial_Z  \omega^m_1\\
V^m_3 &= \Delta_{m3}^{-1} (\partial_Y  \omega^m_1-\partial_X  \omega^m_2),
\end{align*}
where $\Delta_{mi}^{-1}$ is the inverse of the three dimensional Laplace operator
with averaged Dirichlet boundary data $V^m_i$, more precisely
$$\Delta_{mi}^{-1} f
= {2\over {1-\sigma}} \int_\sigma^{1+\sigma\over 2} \int_{B_R} G_R f\> dx \> dR
+ {2\over {1-\sigma}} \int_\sigma^{1+\sigma\over 2} \int_{\partial B_R} V^m_i
\nabla G_R\cdot \nu \> d\mathcal{H}^2,$$
where $G_R$ is Green's function with respect to the Laplace operator in
$B_R$.
From the proof of \cite[Proposition 3.2]{delort}, where
\cite[(3.6)]{delort} holds with
$v^\varepsilon_i$ replaced by $V^m_i$ and $\omega^\varepsilon_i$ replaced by
$\omega^m_i$ but
the remainder terms $w^\varepsilon_i$ given by Greens formula in $B_\sigma$,
we infer that
\begin{align*}
V^m_1 V^m_3 \rightharpoonup V_1 V_3 \textrm{ weakly in } L^2_{\textnormal{loc}}(B_\sigma),\\
V^m_2 V^m_3 \rightharpoonup V_2 V_3 \textrm{ weakly in } L^2_{\textnormal{loc}}(B_\sigma),\\
(V^m_1)^2 + (V^m_2)^2 - (V^m_3)^2 \rightharpoonup (V_1)^2 + (V_2)^2 - (V_3)^2
\textrm{ weakly in } L^2_{\textnormal{loc}}(B_\sigma);
\end{align*}
note that as in \cite{delort} the remainder terms converge strongly in $L^2_{\textnormal{loc}}(B_\sigma)$.

It follows that
\be\label{evm}{1\over {x_1}}\partial_1 v_m \partial_2 v_m \to
 {1\over {x_1}}\partial_1 v_0 \partial_2 v_0\ee
and \[{1\over {x_1}}\left((\partial_1 v_m)^2- (\partial_2 v_m)^2 \right)\to
{1\over {x_1}}\left((\partial_1 v_0)^2-
(\partial_2 v_0)^2\right)\] in the sense of distributions on $B^+_\sigma$ as
$m\to\infty$. Let us remark that in contrast to the true two-dimensional problem,
this alone would {\em not} allow us to pass
to the limit in the domain variation formula for $v_m$!

Observe now that (\ref{vol}) shows that \[\nabla v_m(x)\cdot x -
H(0+) v_m(x)\to 0\] strongly in $L^2_w(B^+_\sigma\setminus
B^+_\tau)$ as $m\to\infty$. It follows that \[\partial_1 v_m x_1 +
\partial_2 v_m x_2\to \partial_1 v_0 x_1 + \partial_2 v_0 x_2\]
 strongly in $L^2_w(B^+_\sigma\setminus B^+_\tau)$ as $m\to\infty$.
But then \begin{align*}
&\int_{B^+_\sigma\setminus B^+_\tau} {1\over {x_1}}(\partial_1 v_m
\partial_1 v_m x_1 +
 \partial_1 v_m \partial_2 v_m x_2)\eta
 \,dx
\\&\to \int_{B^+_\sigma\setminus B^+_\tau} {1\over {x_1}}(\partial_1 v_0 \partial_1
v_0 x_1 + \partial_1 v_0 \partial_2 v_0 x_2)\eta
 \,dx\end{align*}
for each $\eta \in C^0_0(B^+_\sigma\setminus \overline B^+_\tau)$ as
$m\to\infty$. Using (\ref{evm}), we obtain that
\[\int_{B^+_\sigma\setminus B^+_\tau} (\partial_1 v_m)^2 \eta
 \,dx
\to \int_{B^+_\sigma\setminus B^+_\tau} (\partial_1 v_0)^2
\eta
 \,dx\]
for each $0\le \eta \in C^0_0(B^+_\sigma\setminus \overline
B^+_\tau)$
 as $m\to\infty$.
Using once more (\ref{evm}) yields that $\nabla v_m$ converges
strongly in $L^2_{w,\textnormal{loc}}(B^+_\sigma\setminus \overline
B^+_\tau)$. Since $\sigma$ and $\tau$ with $0<\tau<\sigma<1$
were arbitrary, it follows that $\nabla v_m$ converges to $\nabla
v_0$ strongly in $L^2_{w,\textnormal{loc}}(B^+_1\setminus \{ 0\})$.

As a consequence of the strong convergence, we see that \[\int_{B^+_1}{1\over {x_1}}
\nabla (\eta v_0)\cdot \nabla v_0 =0 \quad\text{ for all } \eta \in
C^1_0(B^+_1\setminus\{0\}).\] Combined with the fact that $v_0=0$ in
$B^+_1\cap\{x_2\leq 0\}$, this proves that $v_0\Delta v_0=0$ in the
sense of Radon measures on $B^+_1$.

 \end{proof}
\section{Degenerate points}\label{twodimensions}

\begin{theorem}\label{deg2d}
Let $u$ be a weak solution of {\rm
(\ref{strongp})} such that $u=0$ in $x_2\le 0$ and $M^{x_1x_2}(0+)=\int_{B^+_1} x_1x^+_2\,dx$,
let the free boundary $\partial\{u>0\}\cap B_1^+$
be a continuous injective curve $\sigma=(\sigma_1,\sigma_2)$
such that $\sigma(0)=0$. Then
$\sigma_1(t)\ne 0$ in $[0,t_1)\setminus \{ 0\}$,
$$ \lim_{t\to 0} \frac{\sigma_2(t)}{\sigma_1(t)} = 0$$
and
\[\frac{u(rx)}{\sqrt{r^{-1}\int_{\partial B^+_{r}(0)} u^2
\dhone}}\to
\frac{x_1^2 x_2}{\sqrt{\int_{\partial B^+_{1}(0)} x_1^4 x_2^2
\dhone}}\quad\text{as }r\to 0+,\]
strongly in $W^{1,2}_{w,\textnormal{loc}}(B^+_1\setminus\{0\})$ and
weakly in $W^{1,2}(B^+_1)$.
Moreover,
\begin{align*}
&\frac{u(rx)}{r^\alpha} \to 0 \textrm{ in } L^2_w(B_1^+) \textrm{ for } \alpha \in (0,2) \textrm{ and}\\
&\frac{u(rx)}{r^\alpha}\textrm{ is unbounded in } L^2_w(B_1^+) \textrm{ for } \alpha>2.
\end{align*} 
\end{theorem}

\begin{proof} Let $r_m\to 0+$ be an arbitrary sequence such that the sequence $v_m$
given by (\ref{vm}) converges weakly in $W^{1,2}_w(B^+_1)$ to a limit
$v_0$. By Proposition \ref{blowup} (iii) and Theorem \ref{comp},
$v_0\not\equiv 0$, $v_0$ is homogeneous of degree $H(0+)\ge
5/2$, $v_0$ is continuous, $v_0\ge 0$  and $v_0 \equiv 0$ on $\{ x_1=0\}$ and in $\{
x_2\leq 0\}$, $v_0 \div (\frac{1}{x_1} \nabla v_0)=0$ in $B_1^+$ as a Radon measure, and the
convergence of $v_m$ to $v_0$ is strong in
$W^{1,2}_{w,\textnormal{loc}}(B_1^+\setminus\{0\})$.
 Moreover, the strong
convergence of $v_m$ and the fact proved in Proposition \ref{blowup}
(i) that $V(r_m)\to 0$ as $m\to \infty$ imply that
$$0=\int_{\R^2} \Big( \frac{1}{x_1} {\vert \nabla v_0 \vert}^2 \div\phi
-
2 \nabla u_0 D\phi \nabla v_0 \Big)$$
for every $\phi\in C^1_0(\{ x_1>0\}\cap \{ x_2
>0\};\R^2),$
so that even an analysis in the case of $\{ u=0\}$ consisting of
infinitely many disconnected components (similar to that in \cite{VW})
would be possible in principle. However the structure here is more
complicated. For that reason we confine ourselves to the assumed injective
curve case.

As in the proof of Proposition \ref{2dim}, we will use in each section
of the unit disk where $v_0>0$ the velocity potential $\phi$
defined by
\begin{align*}
\partial_1 \phi = {1\over x_1}\partial_2 v_0, \partial_2 \phi = -{1\over x_1}\partial_1 v_0.
\end{align*}
We obtain that $\phi(\rho\sin\theta,\rho\cos\theta)$
is homogeneous of degree $m=H(0+)\ge 5/2$ and is on the unit circle given by a linear
combination
$f(\cos\theta)=\alpha P_m(\cos \theta)+\beta P_m(-\cos \theta)$,
in the case that the Legendre function $P_m$ and the function $P_m(-x)$ are linearly
independent, and $f(\cos\theta)=\alpha P_m(\cos \theta)+\beta \Re(Q_m(\cos \theta))$
in the case the Legendre function $P_m$ and the function $P_m(-x)$ are linearly
dependent.
Moreover $(1,0)$ is a free boundary point of $v_0$ so that
$f'(0)=0$, which implies $\alpha=\beta$ in the case of linear independence.

On the other hand, Theorem \ref{curve} (ii) implies that
for any ball $\tilde B\subset\subset B^+_1\cap \{ x_2>0\}$,
$v_r=\frac{u(rx)}{\sqrt{r^{-1}\int_{\partial B^+_r(0)} u^2 \dhone}}>0$ in $\tilde B$.
Consequently $\div (\frac{1}{x_1} \nabla v_0)=0$ in $\{ x_1>0\}\cap \{ x_2>0\}$.
However, if there is a free boundary point $x$ in $(0,1)\times (0,1)$ then by homogeneity
the half line connecting that point to the origin consists of free boundary points, so that
$(\div (\frac{1}{x_1} \nabla v_0))(B_\delta(x))>0$ for each $\delta>0$, a contradiction.
Thus $\alpha P_m'+\beta Q_m'$ must be either strictly positive or strictly negative in
$(0,1)$.

In the case $f(\cos\theta)=\alpha (P_m(\cos \theta)+P_m(-\cos \theta))$
we obtain now a contradiction to the fact that $P_m$ is bounded at $1$ and
has a singularity at $-1$.

In the case that $P_m$ is an even function,
we obtain from
$P_m'(0) = m P_{m-1}(0)=\frac{m\sqrt{\pi}}{\Gamma({2-m\over 2})\Gamma({m-1\over 2}+1)}$
and $Q_m'(0) = m Q_{m-1}(0)=-\frac{m\pi^{3/2}\tan(\pi (m-1)/2)}{(m-1)\Gamma({2-m\over 2})\Gamma({m-1\over 2})}$ \\(see http://functions.wolfram.com/07.07.20.0006.01,\\
http://functions.wolfram.com/07.07.03.0001.01, \\http://functions.wolfram.com/07.10.20.0003.01,\\
http://functions.wolfram.com/07.10.03.0001.01),\\
that
$m$ is an even integer $\ge 2$ and that $\beta=0$ so that $f$ is up to a nonzero
multiplicative constant the Legendre polynomial $P_m$.
But, using \cite[Corollary on p. 114]{arnold} there is only one even integer $\ge 2$ such that
$P_m$ has no critical point in $(0,1)$, namely $m=2$.
We obtain $f(x)=c_2 P_2(x) = c_2 {1\over 2} (3x^2-1)$.

In order to obtain the claimed growth we calculate for 
$u_r(x) = u(rx)/r^{\alpha}$ and
a.e. $r\in (0,\delta)$, using 
(\ref{part2}),
\begin{align*}&\left(\int_{\partial B^+_1(0)}\we u_r^2 \dh\right)'
=\frac{2}{r}\left(\int_{B^+_1(0)} \we {\vert \nabla u_r\vert}^2\,dx-\alpha
\int_{\partial B^+_1(0)}\we u_r^2\dh\right)
\\&
\left\{\begin{array}{ll}
\ge \frac{\kappa}{r} \int_{\partial B^+_1(0)}\we u_r^2\dh, &\alpha\in (0,2),\\
\le - \frac{\kappa}{r}  \int_{\partial B^+_1(0)}\we u_r^2\dh, &\alpha>2.
\end{array}\right.
\end{align*}
Integrating we obtain the result.
\end{proof}

\end{document}